\documentclass[11pt,reqno]{amsart}

\usepackage{etex}
\usepackage{amsmath,amssymb,amsthm,amsfonts,mathrsfs, dsfont}
\usepackage[frame,cmtip,arrow,matrix,line,graph,curve]{xy}
\usepackage{graphpap,color,paralist,pstricks}
\usepackage[mathscr]{eucal}
\usepackage{mathabx}
\usepackage[pdftex,colorlinks,backref=page,citecolor=blue]{hyperref}
\usepackage{tikz}
\usepackage{tikz-cd}
\usepackage{epic,eepic}
\usepackage{yfonts}
\usepackage{enumerate}
\usepackage[alphabetic]{amsrefs}
\usepackage{extpfeil}

\usepackage{todonotes}

\theoremstyle{definition}
\newtheorem{theorem}{Theorem}[section]
\newtheorem{definition}[theorem]{Definition}

\newtheorem{lemma}[theorem]{Lemma}
\newtheorem{proposition}[theorem]{Proposition}

\newtheorem*{theorem*}{Theorem}

\theoremstyle{remark}
\newtheorem{remark}[theorem]{Remark}

\renewcommand{\dim}{\operatorname{dim}}

\renewcommand{\H}{\operatorname{H}}

\newcommand{\rk}{\operatorname{rk}}

\renewcommand{\and}{\qquad\text{and}\qquad}

\newcommand{\uJ}{\underline{\operatorname{J}}}

\newcommand{\CH}{\operatorname{CH}}
\newcommand{\uCH}{\underline{\CH}}

\newcommand{\upsi}{\underline{\psi}}
\newcommand{\uvarphi}{\underline{\varphi}}

\renewcommand{\mid}{\hspace{0.3mm}|\hspace{0.3mm}}

\DeclareRobustCommand{\eulerian}{\genfrac<>{0pt}{}}

\title{A semi-small decomposition of the Chow ring of a matroid}

\author{Tom Braden}
\address{Department of Mathematics and Statistics, University of Massachusetts, Amherst, MA.}
\email{braden@math.umass.edu}

\author{June Huh}
\address{School of Mathematics, Institute for Advanced Study, Princeton, NJ.}
\email{junehuh@ias.edu}

\author{Jacob P. Matherne}
\address{Department of Mathematics, University of Oregon, Eugene, OR, and Max-Planck-Institut f\"{u}r Mathematik, Bonn, Germany.}
\email{matherne@uoregon.edu}

\author{Nicholas Proudfoot}
\address{Department of Mathematics, University of Oregon, Eugene, OR.}
\email{njp@uoregon.edu}

\author{Botong Wang}
\address{Department of Mathematics, University of Wisconsin-Madison, Madison, WI.}
\email{wang@math.wisc.edu}

\thanks{
June Huh received support from NSF Grant DMS-1638352 and the Ellentuck Fund.  Jacob Matherne received support from NSF Grant DMS-1638352, the Association of Members of the Institute for Advanced Study, and the Max Planck Institute for Mathematics in Bonn.  Nicholas Proudfoot received support from
NSF Grant DMS-1565036.  Botong Wang received support from NSF Grant DMS-1701305 and the Alfred P. Sloan foundation.}

\begin{document}

\begin{abstract}
We give a semi-small orthogonal decomposition of the Chow ring of a matroid $\mathrm{M}$. 
The decomposition is used to give simple proofs of Poincar\'e duality, the hard Lefschetz theorem, and the Hodge--Riemann relations  for the Chow ring, recovering the main result of \cite{AHK}. 
We also 
show that a similar semi-small orthogonal decomposition holds for the augmented Chow ring of $\mathrm{M}$.
\end{abstract}

\maketitle

\section{Introduction}\label{SectionIntroduction}

A \emph{matroid} $\mathrm{M}$ on a finite set $E$ is a nonempty collection of subsets of $E$, called \emph{flats} of $\mathrm{M}$,
that satisfies the following properties:
\begin{enumerate}[(1)]\itemsep 5pt
\item The intersection of any two flats  is a flat.
\item For any flat $F$, any element in $E \setminus F$ is contained in exactly one flat  that is minimal among the flats strictly containing $F$.
\end{enumerate}
Throughout, we suppose in addition that $\mathrm{M}$ is a \emph{loopless} matroid:
\begin{enumerate}[(1)]\itemsep 5pt
\item[(3)]The empty subset of $E$ is a flat.
\end{enumerate}
We write $\mathscr{L}(\mathrm{M})$ for the lattice of all flats of $\mathrm{M}$.
Every maximal flag of proper flats of $\mathrm{M}$ has the same cardinality $d$, called the \emph{rank} of $\mathrm{M}$.
A matroid can be equivalently defined in terms of its \emph{independent sets} or the \emph{rank function}.
For background  in matroid theory, we refer to \cite{Oxley} and \cite{Welsh}. 


The first aim of the present paper is to decompose the Chow ring of $\mathrm{M}$ as a module over the Chow ring of the deletion $\mathrm{M} \setminus i$ (Theorem \ref{TheoremUnderlinedDecomposition}).
The decomposition resembles the decomposition of the cohomology ring of a projective variety induced by a semi-small map.
In Section \ref{Section4}, we use the decomposition to give simple proofs of Poincar\'e duality, the hard Lefschetz theorem, and the Hodge--Riemann relations  for the Chow ring, recovering the main result of \cite{AHK}.  


The second aim of the present paper is to introduce the augmented Chow ring of $\mathrm{M}$, which contains the graded M\"{o}bius algebra of $\mathrm{M}$ as a subalgebra.  We give an analogous semi-small decomposition of the augmented Chow ring of $\mathrm{M}$ as a module over the augmented Chow ring of the deletion $\mathrm{M} \setminus i$ (Theorem \ref{TheoremDecomposition}), and use this
to prove Poincar\'e duality, the hard Lefschetz theorem, and the Hodge--Riemann relations for the augmented Chow ring.  
These results will play a major role in the forthcoming paper \cite{BHMPW}, where we will prove the Top-Heavy conjecture along with the nonnegativity of the coefficients of the Kazhdan--Lusztig polynomial of a matroid.

\subsection{} 
Let $\underline{S}_\mathrm{M}$  be the ring of polynomials with variables labeled by the nonempty proper flats of $\mathrm{M}$:
\[
\underline{S}_\mathrm{M} \coloneq \mathbb{Q}[x_F\mid \text{$F$ is a nonempty proper flat of $\mathrm{M}$}].
\]
The \emph{Chow ring} of $\mathrm{M}$, introduced by Feichtner and Yuzvinsky in \cite{FY}, is the quotient algebra\footnote{A slightly different presentation for the Chow ring of $\mathrm{M}$ was used in \cite{FY} in a more general context. The present description was used in \cite{AHK}, where the Chow ring of $\mathrm{M}$ was denoted $A(\mathrm{M})$.
 For a comparison of the two presentations, see \cite{BES}.}
\[
\underline{\mathrm{CH}}(\mathrm{M}) \coloneq \underline{S}_\mathrm{M}/(\underline{I}_\mathrm{M}+\underline{J}_\mathrm{M}),
\]
where $\underline{I}_\mathrm{M}$ is the ideal generated by the linear forms
\[
\sum_{i_1 \in F} x_F -\sum_{i_2 \in F} x_F, \ \ \text{for every pair of distinct elements $i_1$ and $i_2$ of $E$},
\]
and $\underline{J}_\mathrm{M}$  is the ideal generated by the quadratic monomials
\[
x_{F_1}x_{F_2}, \ \ \text{for every pair of incomparable nonempty proper flats $F_1$ and $F_2$ of $\mathrm{M}$.}
\]
When $E$ is nonempty, the Chow ring of $\mathrm{M}$ admits a \emph{degree map}  
\[
\underline{\deg}_\mathrm{M}:\underline{\mathrm{CH}}^{d-1}(\mathrm{M}) \longrightarrow \mathbb{Q}, \qquad x_\mathscr{F}\coloneq \prod_{F \in \mathscr{F}} x_{F}\longmapsto 1,
\]
where $\mathscr{F}$ is any complete flag of nonempty proper flats of $\mathrm{M}$ (Definition \ref{DefinitionDegreemap}).
For any integer $k$, the degree map defines the \emph{Poincar\'e pairing}
\[
\underline{\mathrm{CH}}^k(\mathrm{M})  \times \underline{\mathrm{CH}}^{d-k-1}(\mathrm{M})  \longrightarrow \mathbb{Q}, \quad (\eta_1,\eta_2) \longmapsto \underline{\deg}_\mathrm{M}(\eta_1 \eta_2).
\]
If $\mathrm{M}$ is realizable over a field,\footnote{
We say that $\mathrm{M}$ is \emph{realizable} over a field $\mathbb{F}$ if there 
exists a linear subspace $V\subseteq \mathbb{F}^E$ such that $S\subseteq E$ is independent
if and only if the projection from $V$ to $\mathbb{F}^S$ is surjective.
Almost all matroids are not realizable over any field \cite{Nelson}.}
then the Chow ring of $\mathrm{M}$ is isomorphic to the Chow ring of a smooth projective variety over the field (Remark \ref{remark:wonderful}).

Let  $i$ be an element of $E$, and let $\mathrm{M} \setminus i$ be the \emph{deletion} of $i$ from $\mathrm{M}$. 
By definition,  $\mathrm{M} \setminus i$ is the matroid on $E \setminus i$ whose flats are the sets of the form $F \setminus i$ for a flat $F$  of $\mathrm{M}$.
The Chow rings of $\mathrm{M}$ and $\mathrm{M} \setminus i$ are related by the graded algebra homomorphism 
\[
\underline{\theta}_i=\underline{\theta}^\mathrm{M}_i: \underline{\mathrm{CH}}(\mathrm{M} \setminus i) \longrightarrow \underline{\mathrm{CH}}(\mathrm{M}), \qquad x_F \longmapsto x_F + x_{F \cup i},
\]
where a variable in the target is set to zero if its label is not a flat of $\mathrm{M}$. 
Let $\underline{\mathrm{CH}}_{(i)}$ be the image of the homomorphism $\underline{\theta}_i$, and let
$\underline{\mathscr{S}}_i$ be the collection 
\[
\underline{\mathscr{S}}_i=\underline{\mathscr{S}}_i(\mathrm{M})=\big\{F \mid \text{$F$ is a nonempty proper subset of $E \setminus i$ such that $F \in \mathscr{L}(\mathrm{M})$ and $F\cup i \in \mathscr{L}(\mathrm{M})$}\big\}.
\]
The element $i$ is said to be a \emph{coloop} of $\mathrm{M}$ if the ranks of $\mathrm{M}$ and $\mathrm{M} \setminus i$ are not equal.

\begin{theorem}\label{TheoremUnderlinedDecomposition}
If $i$ is not a coloop of $\mathrm{M}$,   there is  a direct sum decomposition of $\underline{\mathrm{CH}}(\mathrm{M}) $ into indecomposable graded $\underline{\mathrm{CH}}(\mathrm{M}\setminus i)$-modules
\begin{equation} \label{eqn:underlined deletion decomposition}
\underline{\mathrm{CH}}(\mathrm{M}) = \underline{\mathrm{CH}}_{(i)} \oplus \bigoplus_{F \in \underline{\mathscr{S}}_i} x_{F\cup i} \underline{\mathrm{CH}}_{(i)}. \tag{$\underline{\mathrm{D}}_1$}
\end{equation}
All pairs of distinct summands are orthogonal for the Poincar\'e pairing of $\underline{\mathrm{CH}}(\mathrm{M})$.
If $i$ is a coloop of $\mathrm{M}$,  there is a direct sum decomposition of $\underline{\mathrm{CH}}(\mathrm{M}) $ into indecomposable graded  $\underline{\mathrm{CH}}(\mathrm{M}\setminus i)$-modules\footnote{When $E = \{i\}$, we treat the symbol $x_\varnothing$ 
as zero in the right-hand side of  \eqref{eqn:underlined deletion decomposition coloop}.}
\begin{equation} \label{eqn:underlined deletion decomposition coloop}
\underline{\mathrm{CH}}(\mathrm{M}) = \underline{\mathrm{CH}}_{(i)} \oplus x_{E\setminus i}\underline{\mathrm{CH}}_{(i)} \oplus \bigoplus_{F\in \underline{\mathscr{S}}_i} x_{F\cup i} \underline{\mathrm{CH}}_{(i)}. \tag{$\underline{\mathrm{D}}_2$}
\end{equation}
All pairs of distinct summands except for the first two are orthogonal for the Poincar\'e pairing of $\underline{\mathrm{CH}}(\mathrm{M})$.
\end{theorem}

We write $\text{rk}_\mathrm{M}:2^E \to \mathbb{N}$ for the rank function of $\mathrm{M}$.
For any proper flat $F$ of $\mathrm{M}$, we set\footnote{The symbols $\mathrm{M}^F$ and $\mathrm{M}_F$
appear inconsistently in the literature, sometimes this way and sometimes interchanged.
The localization is frequently called the restriction.  On the other hand, the contraction is also sometimes called the restriction, especially in the context of hyperplane arrangements,
so we avoid the word restriction to minimize ambiguity.}
\begin{align*}
\mathrm{M}^F&\coloneq\text{the localization
 of $\mathrm{M}$ at $F$, a loopless matroid on $F$ of rank equal to $\text{rk}_\mathrm{M}(F)$},\\
\mathrm{M}_F&\coloneq\text{the contraction of $\mathrm{M}$ by $F$, a loopless matroid on $E \setminus F$  of rank equal to $d-\text{rk}_\mathrm{M}(F)$}.
\end{align*}
The lattice of flats of  $\mathrm{M}^F$ can be identified with the lattice of flats of $\mathrm{M}$ that are contained in $F$, 
and the lattice of flats of  $\mathrm{M}_F$ can be identified with the lattice of flats of $\mathrm{M}$ that contain $F$.
The  $\underline{\CH}(\mathrm{M} \setminus i)$-module summands in the decompositions  (\ref{eqn:underlined deletion decomposition})  and  (\ref{eqn:underlined deletion decomposition coloop}) admit isomorphisms 
\[
\underline{\mathrm{CH}}_{(i)} \cong \underline{\CH}(\mathrm{M} \setminus i)  \ \  \text{and} \ \ 
 x_{F\cup i}\underline{\mathrm{CH}}_{(i)} \cong \underline{\mathrm{CH}}(\mathrm{M}_{F\cup i}) \otimes  \underline{\mathrm{CH}}(\mathrm{M}^{F})[-1], 
\]
where $[-1]$ indicates a degree shift 
(Propositions \ref{DeletionInjection} and \ref{lem:top degree vanishing}).
In addition, if $i$ is a coloop of $\mathrm{M}$, 
\[
x_{E \setminus i}\underline{\mathrm{CH}}_{(i)} \cong \underline{\CH}(\mathrm{M} \setminus i)[-1].
\]
Numerically, the semi-smallness of the decomposition (\ref{eqn:underlined deletion decomposition})  
is reflected in the identity
\[
\dim  x_{F\cup i}\underline{\mathrm{CH}}^{k-1}_{(i)}  = \dim  x_{F\cup i}\underline{\mathrm{CH}}_{(i)}^{d-k-2} \ \ \text{for $F \in \underline{\mathscr{S}}_i$}.
\]
When $\mathrm{M}$ is the Boolean matroid on $E$, 
the graded dimension of $\underline{\CH}(\mathrm{M})$
is given by the Eulerian numbers $\eulerian{d}{k}$,
and the decomposition (\ref{eqn:underlined deletion decomposition coloop}) specializes to the known quadratic recurrence relation
\[
s_d(t)=s_{d-1}(t)+ t \sum_{k=0}^{d-2} {d-1 \choose k} s_k(t) s_{d-k-1}(t), \qquad s_0(t)=1,
\]
where $s_k(t)$ is the $k$-th Eulerian polynomial \cite[Theorem 1.5]{Petersen}.

\subsection{} 
We also give  similar decompositions  for the augmented Chow ring of $\mathrm{M}$,
which we now introduce.
Let  $S_\mathrm{M}$ be the ring of polynomials in two sets of variables
\[
S_\mathrm{M}\coloneq \mathbb{Q}[y_i \mid \text{$i$ is an element of $E$}] \;\otimes\; \mathbb{Q}[x_F \mid \text{$F$ is a proper flat of $\mathrm{M}$}].
\]
The \emph{augmented Chow ring} of $\mathrm{M}$ is the quotient algebra
\[
\mathrm{CH}(\mathrm{M}) \coloneq S_\mathrm{M}/ (I_\mathrm{M}+J_\mathrm{M}),
\]
where $I_\mathrm{M}$ is the ideal generated by the linear forms
 \[
y_i - \sum_{i \notin F} x_F,  \ \  \text{for every element  $i$ of $E$},
\]
and $J_\mathrm{M}$ is the ideal generated by  the quadratic monomials
\begin{align*}
x_{F_1}x_{F_2}, \ \  &\text{for every pair of incomparable proper flats $F_1$ and $F_2$ of $\mathrm{M}$, and}\\
y_i \hspace{0.5mm} x_F,  \ \  &\text{for every element $i$ of $E$ and every proper flat  $F$ of $\mathrm{M}$ not containing $i$.}
\end{align*}
The augmented Chow ring of $\mathrm{M}$ admits a \emph{degree map}  
\[
\deg_\mathrm{M}:\mathrm{CH}^{d}(\mathrm{M}) \longrightarrow \mathbb{Q}, \qquad x_\mathscr{F}\coloneq \prod_{F \in \mathscr{F}} x_{F}\longmapsto 1,
\]
where $\mathscr{F}$ is any complete flag of proper flats of $\mathrm{M}$ (Definition \ref{DefinitionDegreemap}).
For any integer $k$, the degree map defines the \emph{Poincar\'e pairing}
\[
\mathrm{CH}^k(\mathrm{M})  \times \mathrm{CH}^{d-k}(\mathrm{M})  \longrightarrow \mathbb{Q}, \quad (\eta_1,\eta_2) \longmapsto \deg_\mathrm{M}(\eta_1 \eta_2).
\]
If $\mathrm{M}$ is realizable over a field,
then the augmented Chow ring of $\mathrm{M}$ is isomorphic to the Chow ring of a smooth projective variety over the field (Remark \ref{remark:wonderful}).
The augmented Chow ring  contains the \emph{graded M\"obius algebra} $\mathrm{H}(\mathrm{M})$ (Proposition \ref{PropositionMobiusAlgebra}),
and it is related to the Chow ring of $\mathrm{M}$ by the isomorphism
\begin{equation*}\label{tensorMobius}
\underline{\mathrm{CH}}(\mathrm{M}) \cong \mathrm{CH}(\mathrm{M}) \otimes_{\mathrm{H}(\mathrm{M})} \mathbb{Q}.
\end{equation*}
The $\mathrm{H}(\mathrm{M})$-module structure of $\mathrm{CH}(\mathrm{M})$ will be studied in detail in the forthcoming paper \cite{BHMPW}.

As before, we write $\mathrm{M} \setminus i$ for the matroid obtained from $\mathrm{M}$ by deleting the element $i$.
The augmented Chow rings of $\mathrm{M}$ and $\mathrm{M} \setminus i$ are related by the  graded algebra homomorphism 
\[
\theta_i=\theta^\mathrm{M}_i: \mathrm{CH}(\mathrm{M} \setminus i) \longrightarrow \mathrm{CH}(\mathrm{M}), \qquad x_F \longmapsto x_F + x_{F \cup i},
\]
where a variable in the target is set to zero if its label is not a flat of $\mathrm{M}$. 
Let $\mathrm{CH}_{(i)}$ be the image of the homomorphism $\theta_i$, and let
$\mathscr{S}_i$ be the collection 
\[
\mathscr{S}_i=\mathscr{S}_i(\mathrm{M})\coloneq\big\{F \mid \text{$F$ is a proper subset of $E \setminus i$ such that $F \in \mathscr{L}(\mathrm{M})$ and $F\cup i \in \mathscr{L}(\mathrm{M})$}\big\}.
\]

\begin{theorem}\label{TheoremDecomposition}
If $i$ is not a coloop of $\mathrm{M}$, there is  a direct sum decomposition of $\mathrm{CH}(\mathrm{M}) $ into indecomposable graded $\mathrm{CH}(\mathrm{M}\setminus i)$-modules
\begin{equation}\label{eqn:deletion decomposition}
\mathrm{CH}(\mathrm{M}) = \mathrm{CH}_{(i)} \oplus \bigoplus_{F \in \mathscr{S}_i} x_{F\cup i} \mathrm{CH}_{(i)}. \tag{$\mathrm{D}_1$}
\end{equation}
All pairs of distinct summands are orthogonal for the Poincar\'e pairing of $\mathrm{CH}(\mathrm{M})$.
If $i$ is a coloop of $\mathrm{M}$,  there is a direct sum decomposition of $\mathrm{CH}(\mathrm{M}) $ into indecomposable graded  $\mathrm{CH}(\mathrm{M}\setminus i)$-modules
\begin{equation}\label{eqn:deletion decomp coloop}
\mathrm{CH}(\mathrm{M}) = \mathrm{CH}_{(i)} \oplus x_{E\setminus i}\mathrm{CH}_{(i)} \oplus \bigoplus_{F\in \mathscr{S}_i} x_{F\cup i} \mathrm{CH}_{(i)}. \tag{$\mathrm{D}_2$}
\end{equation}
All pairs of distinct summands except for the first two are orthogonal for the Poincar\'e pairing of $\mathrm{CH}(\mathrm{M})$.
\end{theorem}

The  $\CH(\mathrm{M} \setminus i)$-module summands in the decompositions  (\ref{eqn:deletion decomposition})  and  (\ref{eqn:deletion decomp coloop}) admit isomorphisms 
\[
\mathrm{CH}_{(i)} \cong \CH(\mathrm{M} \setminus i) \ \  \text{and} \ \ 
 x_{F\cup i}\mathrm{CH}_{(i)} \cong \underline{\mathrm{CH}}(\mathrm{M}_{F\cup i}) \otimes  \mathrm{CH}(\mathrm{M}^{F})[-1], 
\]
where $[-1]$ indicates a degree shift (Propositions \ref{DeletionInjection} and \ref{lem:top degree vanishing}).
In addition, if $i$ is a coloop of $\mathrm{M}$, 
\[
x_{E \setminus i}\mathrm{CH}_{(i)} \cong \CH(\mathrm{M} \setminus i)[-1].
\]
Numerically, the semi-smallness of the decomposition (\ref{eqn:deletion decomposition})  
is reflected in the identity
\[
\dim  x_{F\cup i}\mathrm{CH}^{k-1}_{(i)}  = \dim  x_{F\cup i}\mathrm{CH}_{(i)}^{d-k-1}\ \ \text{for $F \in \mathscr{S}_i$}.
\]

\subsection{} 
Let $\mathrm{B}$ be the Boolean matroid on $E$.
By definition, every subset of $E$ is a flat of $\mathrm{B}$.
The Chow rings of $\mathrm{B}$ and $\mathrm{M}$ are related by
the surjective graded algebra homomorphism
\[
\underline{\mathrm{CH}}(\mathrm{B}) \longrightarrow \underline{\mathrm{CH}}(\mathrm{M}), \qquad x_S \longmapsto x_S,
\]
where a variable in the target is set to zero if its label is not a flat of $\mathrm{M}$. Similarly,
we have a surjective graded algebra homomorphism
\[
\mathrm{CH}(\mathrm{B}) \longrightarrow \mathrm{CH}(\mathrm{M}), \qquad x_S \longmapsto x_S,
\]
where a variable in the target is set to zero if its label is not a flat of $\mathrm{M}$.
As in \cite[Section 4]{AHK}, we may identify the Chow ring $\underline{\mathrm{CH}}(\mathrm{B})$ with the ring of piecewise polynomial functions modulo linear functions on the normal fan  $\underline{\Pi}_\mathrm{B}$ of the standard permutohedron in $\mathbb{R}^E$.
Similarly, the augmented Chow ring $\mathrm{CH}(\mathrm{B})$ can be identified 
with the ring of piecewise polynomial functions modulo linear functions of the normal fan $\Pi_\mathrm{B}$ of the stellahedron in $\mathbb{R}^E$ (Definition \ref{DefinitionAugmentedBergmanFan}).
A convex piecewise linear function on a complete fan is said to be \emph{strictly convex} if there is a bijection between the cones in the fan and the faces of the graph of the function.

In Section \ref{Section4}, we use Theorems \ref{TheoremUnderlinedDecomposition} and \ref{TheoremDecomposition} to give simple proofs of Poincar\'e duality, the hard Lefschetz theorem, and the Hodge--Riemann relations for  $\underline{\mathrm{CH}}(\mathrm{M})$ and $\mathrm{CH}(\mathrm{M})$.

\begin{theorem}\label{TheoremChowKahlerPackage}
Let $\underline{\ell}$ be a strictly convex piecewise linear function on $\underline{\Pi}_\mathrm{B}$, viewed as an element of $\underline{\mathrm{CH}}^1(\mathrm{M})$.
\begin{enumerate}[(1)]\itemsep 5pt
\item (Poincar\'e duality theorem) For every nonnegative integer $k < \frac{d}{2}$, the bilinear pairing
\[
\underline{\mathrm{CH}}^k(\mathrm{M})  \times \underline{\mathrm{CH}}^{d-k-1}(\mathrm{M})  \longrightarrow \mathbb{Q}, \quad (\eta_1,\eta_2) \longmapsto \underline{\deg}_\mathrm{M}(\eta_1 \eta_2)
\]
is non-degenerate.
\item (Hard Lefschetz theorem)  For every nonnegative integer $k < \frac{d}{2}$, the multiplication map
\[
\underline{\mathrm{CH}}^k(\mathrm{M})  \longrightarrow  \underline{\mathrm{CH}}^{d-k-1}(\mathrm{M}), \quad \eta \longmapsto \underline{\ell}^{d-2k-1}  \eta
\]
is an isomorphism.
\item (Hodge--Riemann relations) For every nonnegative integer $k <\frac{d}{2}$, the bilinear form
\[
\underline{\mathrm{CH}}^k(\mathrm{M})  \times \underline{\mathrm{CH}}^{k}(\mathrm{M})  \longrightarrow \mathbb{Q}, \quad (\eta_1,\eta_2) \longmapsto (-1)^k \underline{\deg}_\mathrm{M}(\underline{\ell}^{d-2k-1} \eta_1 \eta_2)
\]
is positive definite on the kernel of multiplication by $\underline{\ell}^{d-2k}$.
\end{enumerate}
Let $\ell$ be a strictly convex piecewise linear function on $\Pi_\mathrm{B}$,  viewed as an element of $\mathrm{CH}^1(\mathrm{M})$.
\begin{enumerate}[(1)]\itemsep 5pt
\item[(4)] (Poincar\'e duality theorem) For every nonnegative integer $k \le \frac{d}{2}$, the bilinear pairing
\[
\mathrm{CH}^k(\mathrm{M})  \times \mathrm{CH}^{d-k}(\mathrm{M})  \longrightarrow \mathbb{Q}, \quad (\eta_1,\eta_2) \longmapsto \deg_\mathrm{M}(\eta_1 \eta_2)
\]
is non-degenerate.
\item[(5)] (Hard Lefschetz theorem)  For every nonnegative integer $k \le \frac{d}{2}$, the multiplication map
\[
\mathrm{CH}^k(\mathrm{M})  \longrightarrow  \mathrm{CH}^{d-k}(\mathrm{M}), \quad \eta \longmapsto \ell^{d-2k} \eta
\]
is an isomorphism.
\item[(6)] (Hodge--Riemann relations) For every nonnegative integer $k \le \frac{d}{2}$, the bilinear form
\[
\mathrm{CH}^k(\mathrm{M})  \times \mathrm{CH}^{k}(\mathrm{M})  \longrightarrow \mathbb{Q}, \quad (\eta_1,\eta_2) \longmapsto (-1)^k \deg_\mathrm{M}( \ell^{d-2k} \eta_1 \eta_2)
\]
is positive definite on the kernel of multiplication by $\ell^{d-2k+1}$.
\end{enumerate}
\end{theorem}

 Theorem \ref{TheoremChowKahlerPackage} holds non-vacuously, as
  there are strictly convex piecewise linear functions on $\Pi_\mathrm{B}$ and $\underline{\Pi}_\mathrm{B}$ (Proposition \ref{PropositionBoolean}).
The first part of Theorem \ref{TheoremChowKahlerPackage} on $\uCH(\mathrm{M})$ recovers the main result of \cite{AHK}.\footnote{Independent proofs of Poincar\'e duality  for $\underline{\mathrm{CH}}(\mathrm{M})$ were  given in  \cite{BES} and \cite{BDF}. The authors of \cite{BES} also prove the degree $1$ Hodge--Riemann relations for $\underline{\mathrm{CH}}(\mathrm{M})$.}
The second part of Theorem  \ref{TheoremChowKahlerPackage} on $\CH(\mathrm{M})$ is new.

\subsection{} 
In Section \ref{Section5},  we use Theorems \ref{TheoremUnderlinedDecomposition} and \ref{TheoremDecomposition}
to obtain decompositions of $\underline{\CH}(\mathrm{M})$ and $\CH(\mathrm{M})$ related to those appearing in \cite[Theorem 6.18]{AHK}.
Let $\underline{\mathrm{H}}_{\underline{\alpha}}(\mathrm{M})$ be the subalgebra of $\uCH(\mathrm{M})$ generated by the element
\[
\underline{\alpha}_{\mathrm{M}}\coloneq  \sum_{i \in G} x_G \in \underline{\mathrm{CH}}^1(\mathrm{M}),
\]
where the sum is over all nonempty proper flats $G$ of $\mathrm{M}$ containing a given element $i$ in $E$,
and let $\H_\alpha(\mathrm{M})$ be the subalgebra of $\CH(\mathrm{M})$ generated by the element
\[
\alpha_{\mathrm{M}} \coloneq \sum_G x_G \in \mathrm{CH}^1(\mathrm{M}),
\]
where the sum is over all proper flats $G$ of $\mathrm{M}$.
We define graded subspaces $\underline{\mathrm{J}}_{\underline{\alpha}}(\mathrm{M})$ and $\mathrm{J}_\alpha(\mathrm{M})$ by 
\[
\underline{\mathrm{J}}_{\underline{\alpha}}^k(\mathrm{M})\coloneq \begin{cases} \underline{\mathrm{H}}_{\underline{\alpha}}^k(\mathrm{M}) & \text{if $k \neq d-1$,} \\ \hfil 0 & \text{if $k=d-1$,} \end{cases}
\qquad
\mathrm{J}_\alpha^k(\mathrm{M}) \coloneq \begin{cases} \H_\alpha^k(\mathrm{M}) & \text{if $k \neq d$,} \\ \hfil 0 & \text{if $k=d$.} \end{cases}
\]
A degree computation shows that the elements $\underline{\alpha}_\mathrm{M}^{d-1}$ and 
 $\alpha_\mathrm{M}^d$ are nonzero (Proposition \ref{PropositionAlphaDegree}).

\begin{theorem}\label{TheoremSimplexDecomposition}
Let $\mathscr{C} = \mathscr{C}(\mathrm{M})$ be the set of all nonempty proper flats of $\mathrm{M}$, and let $\underline{\mathscr{C}} = \underline{\mathscr{C}}(\mathrm{M})$ be the set of all proper flats of $\mathrm{M}$ with rank at least two. 
\begin{enumerate}[(1)]\itemsep 5pt
\item We have a decomposition of $\underline{\mathrm{H}}_{\underline{\alpha}}(\mathrm{M})$-modules
\begin{equation}\label{underlinedalphadecomposition}
\underline{\CH}(\mathrm{M}) = \underline{\mathrm{H}}_{\underline{\alpha}}(\mathrm{M}) \oplus \ \bigoplus_{F \in \underline{\mathscr{C}}}  \ \underline{\psi}^F_\mathrm{M}\ \uCH(\mathrm{M}_F)\otimes \uJ_{\underline{\alpha}}(\mathrm{M}^F). \tag{$\underline{\mathrm{D}}_3$}
\end{equation}
All pairs of distinct summands are orthogonal for the Poincar\'e pairing of $\underline{\mathrm{CH}}(\mathrm{M})$.
\item We have a decomposition of $\mathrm{H}_\alpha(\mathrm{M})$-modules
\begin{equation}\label{alphadecomposition}
\CH(\mathrm{M}) = \H_\alpha(\mathrm{M}) \oplus \bigoplus_{F \in \mathscr{C}}  \psi^F_\mathrm{M} \ \uCH(\mathrm{M}_F)\otimes {\mathrm{J}}_\alpha(\mathrm{M}^F). \tag{$\mathrm{D}_3$}
\end{equation}
All pairs of distinct summands are orthogonal for the Poincar\'e pairing of $\mathrm{CH}(\mathrm{M})$.
\end{enumerate}
\end{theorem}

Here $\underline{\psi}^F_\mathrm{M}$   is
the injective $\underline{\CH}(\mathrm{M})$-module homomorphism (Propositions \ref{DefinitionUnderlinedPush} and \ref{upsi injective})
\[
\underline{\psi}^F_\mathrm{M}: \underline{\CH}(\mathrm{M}_F) \otimes \underline{\CH}(\mathrm{M}^F)  \longrightarrow \underline{\CH}(\mathrm{M}),
\quad \prod_{F'} x_{F'\setminus F} \otimes \prod_{F''} x_{F''} \longmapsto x_F \prod_{F'} x_{F'} \prod_{F''} x_{F''},
\]
and  $\psi^F_\mathrm{M}$ is the injective $\CH(\mathrm{M})$-module homomorphism (Propositions \ref{DefinitionXPushforward} and \ref{PropositionPushforwardI})
\[
\psi^F_\mathrm{M}: \underline{\CH}(\mathrm{M}_F) \otimes \CH(\mathrm{M}^F)  \longrightarrow \CH(\mathrm{M})
\quad \prod_{F'} x_{F'\setminus F} \otimes \prod_{F''} x_{F''} \longmapsto x_F \prod_{F'} x_{F'} \prod_{F''} x_{F''}.
\]
When $\mathrm{M}$ is the Boolean matroid on $E$, the decomposition (\ref{underlinedalphadecomposition}) specializes to a linear recurrence relation for the Eulerian polynomials
\[
0=1+\sum_{k=0}^d {d \choose k} \frac{t-t^{d-k}}{1-t} s_{k}(t), \qquad s_0(t)=1.
\]
When applied repeatedly, Theorem \ref{TheoremSimplexDecomposition} produces bases of $\underline{\CH}(\mathrm{M})$ and $\CH(\mathrm{M})$ that are permuted by the automorphism group of $\mathrm{M}$.\footnote{Different bases of $\underline{\CH}(\mathrm{M})$ are given in \cite[Corollary 1]{FY} and \cite[Corollary 3.3.3]{BES}.}

\noindent
{\bf Acknowledgments.}
We thank Christopher Eur and Matthew Stevens for useful discussions.

\section{The Chow ring and the augmented Chow ring of a matroid}\label{Section2}

In this section, we collect the various properties of the algebras $\underline{\mathrm{CH}}(\mathrm{M})$ and $\mathrm{CH}(\mathrm{M})$
that we will need in order to prove Theorems 
\ref{TheoremUnderlinedDecomposition}--\ref{TheoremSimplexDecomposition}.
In Section \ref{sec:fans}, we review the definition and basic properties of the \emph{Bergman fan} and introduce the closely related 
\emph{augmented Bergman fan} of a matroid.
Section \ref{sec:stars} is devoted to understanding the stars of the various rays in these two fans, while Section \ref{sec:weights}
is where we compute the space of balanced top-dimensional weights on each fan.
Feichtner and Yuzvinsky showed that the Chow ring of a matroid coincides with the Chow ring of the toric variety
associated with its Bergman fan \cite[Theorem 3]{FY}, and we establish the analogous result for the augmented Chow ring in Section \ref{sec:rings}.
Section \ref{sec:mobius} is where we show that the augmented Chow ring contains the graded M\"obius algebra.
In Section \ref{sec:gysin}, we use the results of Section \ref{sec:stars} to construct various homomorphisms that relate
the Chow and augmented Chow rings of different matroids. 

\begin{remark}
It is worth noting why we need to interpret $\underline{\mathrm{CH}}(\mathrm{M})$ and $\mathrm{CH}(\mathrm{M})$
as Chow rings of toric varieties.
First, the study of balanced weights on the Bergman fan and augmented Bergman fan allow us to show 
that $\underline{\mathrm{CH}}^{d-1}(\mathrm{M})$ and $\mathrm{CH}^d(\mathrm{M})$ are nonzero, which is not easy to
prove directly from the definitions.
The definition of the pullback and pushforward maps in Section \ref{sec:gysin} is made cleaner by thinking about fans, though it would also
be possible to define these maps by taking Propositions \ref{DefinitionXPullback}, \ref{DefinitionXPushforward}, \ref{DefinitionUnderlinedPull}, \ref{DefinitionUnderlinedPush}, \ref{DefinitionYPull}, and \ref{DefinitionYPush} as definitions. 
Finally, and most importantly, the fan perspective will be essential for understanding
the ample classes that appear in Theorem \ref{TheoremChowKahlerPackage}.
\end{remark}

\subsection{Fans}\label{sec:fans}

Let $E$ be a finite set, and let $\mathrm{M}$ be a loopless matroid of rank $d$ on the ground set $E$.
We write $\text{rk}_\mathrm{M}$ for the rank function of $\mathrm{M}$,
and write $\text{cl}_\mathrm{M}$ for the closure operator of $\mathrm{M}$, which for a set $S$ returns the smallest flat containing $S$.  
The \emph{independence complex} $\mathrm{I}_\mathrm{M}$ of $\mathrm{M}$ is the simplicial complex of independent sets of $\mathrm{M}$.  A set $I\subseteq E$ is independent if and only if the rank of $\text{cl}_{\mathrm{M}}(I)$ is $|I|$.
The vertices of $\mathrm{I}_\mathrm{M}$  are the elements of the ground set $E$, and
a collection of vertices  is a face of $\mathrm{I}_\mathrm{M}$ when the corresponding set of elements is an independent set of $\mathrm{M}$.
The \emph{Bergman complex} $\underline{\Delta}_\mathrm{M}$ of $\mathrm{M}$ is the order complex of the poset of nonempty proper flats of $\mathrm{M}$.
The vertices of $\underline{\Delta}_\mathrm{M}$ are the nonempty proper flats of $\mathrm{M}$, and
a collection of vertices   is a face  of $\underline{\Delta}_\mathrm{M}$ when the corresponding set of flats is a flag.
The independence complex of $\mathrm{M}$ is pure of dimension $d-1$,
and the Bergman complex of $\mathrm{M}$ is pure of dimension $d-2$.
For a detailed study of the simplicial complexes  $\mathrm{I}_\mathrm{M}$ and $\underline{\Delta}_\mathrm{M}$, we refer to  \cite{Bjorner}.
We introduce the \emph{augmented Bergman complex} $\Delta_\mathrm{M}$ of $\mathrm{M}$ as a simplicial complex that interpolates
between the independence complex and the Bergman complex of $\mathrm{M}$.

\begin{definition}
Let $I$ be an independent set of $\mathrm{M}$, and let $\mathscr{F}$ be a flag of proper flats of $\mathrm{M}$.
When $I$ is contained in every flat in $\mathscr{F}$, we say that $I$ is \emph{compatible} with $\mathscr{F}$ and write $I \le \mathscr{F}$.
The \emph{augmented Bergman complex} $\Delta_\mathrm{M}$ of $\mathrm{M}$ is the simplicial complex of 
all compatible pairs  $I \le \mathscr{F}$,
where $I$ is an independent set of $\mathrm{M}$ and $\mathscr{F}$ is a flag of proper flats of $\mathrm{M}$.
\end{definition}

A vertex of the augmented Bergman complex $\Delta_\mathrm{M}$ is either a singleton subset of $E$ or a proper flat of $\mathrm{M}$.
More precisely, the vertices of $\Delta_\mathrm{M}$ are the compatible pairs either of the form $\{i\} \le \varnothing$ or of the 
form $\varnothing \le \{F\}$, where
$i$ is an element of $E$ and $F$ is a proper flat of $\mathrm{M}$.
The augmented Bergman complex  contains both the independence complex $\mathrm{I}_\mathrm{M}$
and the Bergman complex $\underline{\Delta}_\mathrm{M}$ as subcomplexes.
In fact, $\Delta_\mathrm{M}$  contains the order complex of the poset of proper flats of $\mathrm{M}$, which is
the cone over  the Bergman complex with the cone point corresponding to the empty flat.
It is straightforward to check that $\Delta_\mathrm{M}$ is pure of dimension $d-1$.

\begin{proposition}\label{PropositionConnected}
The Bergman complex and the augmented Bergman complex of $\mathrm{M}$ are both connected in codimension $1$.
\end{proposition}

\begin{proof}
The statement about the Bergman complex is a direct consequence of its shellability \cite{Bjorner}.
We prove the statement about the augmented Bergman complex using the statement about the Bergman complex.

The claim is that, given any two facets of $\Delta_\mathrm{M}$, one may travel from one facet to the other by passing through faces of codimension at most $1$.
Since the Bergman complex of $\mathrm{M}$ is connected in codimension $1$, the subcomplex of $\Delta_\mathrm{M}$ consisting of faces of the form $\varnothing \le \mathscr{F}$ is connected in codimension $1$.
Thus it suffices to show that any  facet of  $\Delta_\mathrm{M}$ can be connected to a facet of the form $\varnothing\le \mathscr{F}$ through codimension $1$ faces.

Let $I \le \mathscr{F}$ be a facet of $\Delta_{\mathrm{M}}$.
If $I$ is nonempty, choose any element $i$ of $I$, and consider the flag of flats $\mathscr{G}$ obtained by adjoining the closure of $I\setminus i$  to $\mathscr{F}$.
The independent set $I \setminus i$ is compatible with the flag $\mathscr{G}$, and the facet $I \le \mathscr{F}$ is adjacent to the facet
$I\setminus i \le \mathscr{G}$.
Repeating the procedure,
we can connect the given facet to a facet of the desired form through codimension $1$ faces.
\end{proof}


Let $\mathbb{R}^E$ be the vector space spanned  by the standard basis vectors $\mathbf{e}_i$ corresponding to the elements $i \in E$.
For an arbitrary subset $S \subseteq E$, we set
\[
\mathbf{e}_S \coloneq \sum_{i \in S}\mathbf{e}_i.
\]
For an element $i \in E$, we write $\rho_i$ for the ray generated by the vector $\mathbf{e}_i$ in $\mathbb{R}^E$.  
For a subset $S\subseteq E$, we write $\rho_S$ for the ray generated by the vector $-\mathbf{e}_{E \setminus S}$ in $\mathbb{R}^E$,
and write $\underline{\rho}_S$ for the ray generated by the vector $\mathbf{e}_S$ in $\mathbb{R}^E/\langle \mathbf{e}_S\rangle$.
Using these rays, we construct  fan models of the Bergman complex and the augmented Bergman complex  as follows.

\begin{definition}\label{DefinitionAugmentedBergmanFan}
The \emph{Bergman fan} $\underline{\Pi}_\mathrm{M}$ of $\mathrm{M}$ is a simplicial fan in the quotient space $\mathbb{R}^E / \langle \mathbf{e}_E \rangle$
with rays $\underline{\rho}_F$  for nonempty proper flats $F$ of $\mathrm{M}$.
The cones of $\underline{\Pi}_\mathrm{M}$  are of the form
\[
\underline{\sigma}_\mathscr{F} \coloneq \text{cone}\{ \mathbf{e}_F\}_{F \in \mathscr{F}} =\text{cone}\{ -\mathbf{e}_{E \setminus F}\}_{F \in \mathscr{F}},
\]
where $\mathscr{F}$ is a flag of nonempty proper flats of $\mathrm{M}$.

The \emph{augmented Bergman fan} $\Pi_\mathrm{M}$ of $\mathrm{M}$
is a simplicial fan in $\mathbb{R}^E$ with rays $\rho_i$ for elements $i$ in $E$ and $\rho_F$ for proper flats $F$ of $\mathrm{M}$.
The cones of the augmented Bergman fan are of the form
\[
\sigma_{I \le \mathscr{F}} \coloneq \text{cone}\{\mathbf{e}_i\}_{i \in I}+\text{cone}\{ -\mathbf{e}_{E \setminus F}\}_{F \in \mathscr{F}},
\]
where $\mathscr{F}$ is a flag of proper flats  of $\mathrm{M}$
and $I$ is an independent set of $\mathrm{M}$ compatible with $\mathscr{F}$.
We write $\sigma_I$ for the cone $\sigma_{I \le \mathscr{F}}$ when $\mathscr{F}$ is the empty flag of flats of $\mathrm{M}$. 
\end{definition}

\begin{remark}\label{star}
If $E$ is nonempty, then
the Bergman fan $\underline{\Pi}_{\mathrm{M}}$ is the star of the ray $\rho_\varnothing$ in the augmented Bergman fan $\Pi_{\mathrm{M}}$.
If $E$ is empty, then $\underline{\Pi}_{\mathrm{M}}$ and $\Pi_{\mathrm{M}}$ both consist of a single $0$-dimensional cone.
\end{remark}

\begin{figure}[h]
\begin{tikzpicture}
\node[circle,fill,inner sep=1.5pt] (center) at (0,0) {};
\draw[thick] (0,-3.5) -- (0,-0.125); 
\draw[thick] (0,0.125) -- (0,3.5); 
\draw[thick] (-3.5,0) -- (-0.125,0); 
\draw[thick] (0.125,0) -- (3.5,0); 
\draw[thick] (-0.125,-0.125) -- (-3.5,-3.5); 

\node (origin) at (.65,.25) {\Small$\varnothing \le \varnothing$}; 
\node (pxaxis) at (3.25,.25) {\Small$\{1\} \le \varnothing$}; 
\node (nxaxis) at (-3,.25) {\Small$\varnothing \le \{\{2\}\}$}; 
\node (pyaxis) at (.75,3.25) {\Small$\{2\} \le \varnothing$}; 
\node (nyaxis) at (.95,-3.25) {\Small$\varnothing \le \{\{1\}\}$}; 
\node (dldiag) at (-4.15,-3.25) {\Small$\varnothing \le \{\varnothing\}$}; 
\node (firstquad) at (1.75,1.65) {\Small$\{1,2\} \le \varnothing$}; 
\node (secondquad) at (-1.75, 1.65) {\Small$\{2\} \le \{\{2\}\}$}; 
\node (thirdquadtop) at (-3, -1.35) {\Small$\varnothing \le \{\varnothing,\{2\}\}$}; 
\node (thirdquadbot) at (-1.2,-2.45) {\Small$\varnothing \le \{\varnothing,\{1\}\}$}; 
\node (fourthquad) at (1.75,-1.65) {\Small$\{1\} \le \{\{1\}\}$}; 
\end{tikzpicture}
\caption{The augmented Bergman fan of the rank $2$ Boolean matroid on $\{1,2\}$.}
\end{figure}
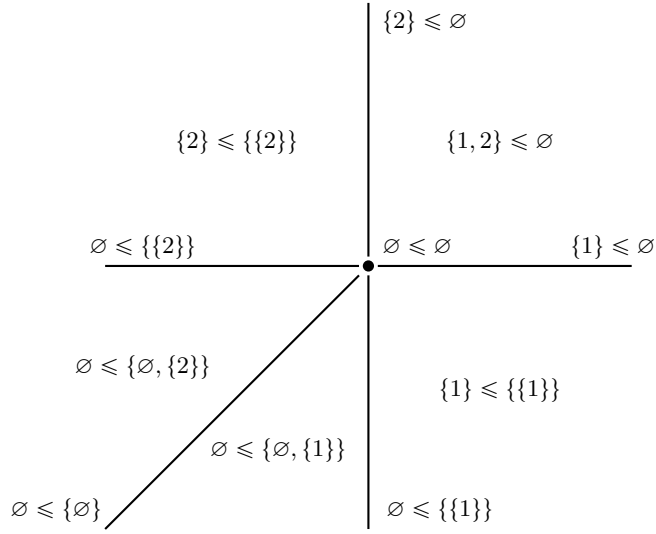

Let $\mathrm{N}$ be another loopless matroid on $E$.
The matroid $\mathrm{M}$ is said to be a \emph{quotient} of $\mathrm{N}$ if every flat of $\mathrm{M}$ is a flat of $\mathrm{N}$.
The condition implies that every independent set of $\mathrm{M}$ is an independent set of $\mathrm{N}$ \cite[Proposition 8.1.6]{Kung}.
Therefore, 
when $\mathrm{M}$ is a quotient of  $\mathrm{N}$, 
the augmented Bergman fan of $\mathrm{M}$ is a subfan of the augmented Bergman fan of $\mathrm{N}$,
and the  Bergman fan of $\mathrm{M}$ is a subfan of the Bergman fan of $\mathrm{N}$.
In particular, 
we have inclusions of fans
 $\Pi_\mathrm{M} \subseteq \Pi_{\mathrm{B}}$
and  $\underline{\Pi}_\mathrm{M} \subseteq \underline{\Pi}_{\mathrm{B}}$,
where $\mathrm{B}$ is the \emph{Boolean matroid} on $E$ defined by the condition that $E$ is an independent set of $\mathrm{B}$.

\begin{proposition}\label{PropositionBoolean}
The Bergman fan and the augmented Bergman fan of $\mathrm{B}$ are each normal fans of convex polytopes.
In particular, there are strictly convex piecewise linear functions on $\underline{\Pi}_\mathrm{B}$ and $\Pi_\mathrm{B}$.
\end{proposition}

The above proposition can be used to show that the augmented Bergman fan  and the Bergman fan of $\mathrm{M}$ are, in fact, fans.

\begin{proof}
The statement for the Bergman fan is well-known: The Bergman fan of $\mathrm{B}$ is the normal fan of the standard permutohedron in 
$\mathbf{e}_E^\perp \subseteq \mathbb{R}^E$.
See, for example, \cite[Section 2]{AHK}. 
The statement for the augmented Bergman fan $\Pi_{\mathrm{B}}$ follows from the fact that it is an iterated stellar subdivision of the normal fan of the  simplex
\[
\text{conv}\{\mathbf{e}_i,\mathbf{e}_E\}_{i \in E} \subseteq \mathbb{R}^E.
\]
More precisely, $\Pi_\mathrm{B}$ is isomorphic to the fan $\Sigma_\mathscr{P}$ in \cite[Definition 2.3]{AHK},
where $\mathscr{P}$ is the order filter of all subsets of $E \cup 0$ containing the new element $0$, via the linear isomorphism
\[
\mathbb{R}^E \longrightarrow \mathbb{R}^{E \cup 0}/\langle \mathbf{e}_{E}+\mathbf{e}_0\rangle, \quad \mathbf{e}_j \longmapsto \mathbf{e}_j.
\]
It is shown in \cite[Proposition 2.4]{AHK} that $\Sigma_\mathscr{P}$ is an iterated stellar subdivision of the normal fan of the simplex.\footnote{In fact, the augmented Bergman fan $\Pi_\mathrm{B}$ is the normal fan of the stellahedron in $\mathbb{R}^E$, the graph associahedron of the star graph with $|E|$ endpoints.
We refer to \cite{CD} and \cite{Devadoss} for detailed discussions of graph associahedra and their realizations.}
\end{proof}

A direct inspection shows that $\Pi_\mathrm{M}$ is a \emph{unimodular fan}; that is, the set of primitive ray generators in any cone in $\Pi_\mathrm{M}$ is a subset of a basis of the free abelian group $\mathbb{Z}^E$.
It follows  that $\underline{\Pi}_\mathrm{M}$ is also a unimodular fan;
that is, the set of primitive ray generators in any cone in $\Pi_\mathrm{M}$ is a subset of a basis of  the free abelian group $\mathbb{Z}^E/ \langle \mathbf{e}_E \rangle$.

\subsection{Stars}\label{sec:stars}

For any element $i$ of $E$, we write $\text{cl}(i)$ for the closure of $i$ in $\mathrm{M}$, and write $\iota_i$  for the injective linear map
\[
\iota_i:\mathbb{R}^{E \setminus \text{cl}(i)} \longrightarrow \mathbb{R}^{E}/\langle \mathbf{e}_i \rangle, \qquad \mathbf{e}_j \longmapsto \mathbf{e}_j.
\]
For any proper flat $F$ of $\mathrm{M}$, we write $\iota_F$ for the  linear isomorphism
\[
\iota_F:\mathbb{R}^{E \setminus F} / \langle \mathbf{e}_{E \setminus F} \rangle\ \oplus \ \mathbb{R}^F \longrightarrow \mathbb{R}^E  / \langle \mathbf{e}_{E \setminus F} \rangle,  \qquad \mathbf{e}_j \longmapsto \mathbf{e}_j.
\]
For any nonempty proper flat $F$ of $\mathrm{M}$, we write $\underline{\iota}_F$ for the linear isomorphism
\[
\underline{\iota}_F:\mathbb{R}^{E \setminus F} / \langle \mathbf{e}_{E \setminus F} \rangle\ \oplus \ \mathbb{R}^F/ \langle \mathbf{e}_F \rangle \longrightarrow \mathbb{R}^E  / \langle \mathbf{e}_E, \mathbf{e}_{E \setminus F} \rangle,  \qquad \mathbf{e}_j \longmapsto \mathbf{e}_j.
\]
Let $\mathrm{M}^F$ be the localization of $\mathrm{M}$ at $F$,
and let $\mathrm{M}_F$ be the contraction of $\mathrm{M}$ by $F$.

\begin{proposition}\label{PropositionStars}
The following are descriptions of the stars of the rays in $\Pi_\mathrm{M}$ and $\underline{\Pi}_\mathrm{M}$ using the three linear maps above.
\begin{enumerate}[(1)]\itemsep 5pt
\item\label{PropositionStars1} For any element $i\in E$, the linear map $\iota_i$ identifies the augmented Bergman fan of $\mathrm{M}_{\text{cl}(i)}$ with the star of  the ray $\rho_i$ in the augmented Bergman fan of $\mathrm{M}$:
\[
\Pi_{\mathrm{M}_{\text{cl}(i)}} \cong \text{star}_{\rho_i} \Pi_\mathrm{M}.
\]
\item\label{PropositionStars2} For any proper flat $F$ of $\mathrm{M}$, the linear map $\iota_F$ identifies the product of the Bergman fan of $\mathrm{M}_F$ and the augmented Bergman  fan of $\mathrm{M}^F$ with the star of the ray $\rho_F$ in the augmented Bergman fan of $\mathrm{M}$:
\[
\underline{\Pi}_{\mathrm{M}_{F}}\times \Pi_{\mathrm{M}^F} \cong \text{star}_{\rho_F} \Pi_\mathrm{M}.
\]
\item For any nonempty proper flat $F$ of $\mathrm{M}$, the linear map $\underline{\iota}_F$ identifies the product of the Bergman fan of $\mathrm{M}_F$ and the Bergman  fan of $\mathrm{M}^F$ with the star of the ray $\underline{\rho}_F$ in the Bergman fan of $\mathrm{M}$:
\[
\underline{\Pi}_{\mathrm{M}_{F}}\times \underline{\Pi}_{\mathrm{M}^F} \cong \text{star}_{\underline{\rho}_F} \underline{\Pi}_\mathrm{M}.
\]
  \end{enumerate}
\end{proposition}

Repeated applications of the first statement 
show that, for any independent set $I$ of $\mathrm{M}$,
the star of the cone $\sigma_{I}$ in $\Pi_\mathrm{M}$ can be identified with the augmented Bergman fan of $\mathrm{M}_{\text{cl}(I)}$,
where $\text{cl}(I)$ is the closure of $I$ in $\mathrm{M}$.

\begin{proof}
The first statement follows from the following facts: A flat of $\mathrm{M}$ contains $i$ if and only if it contains $\text{cl}(i)$,
and an independent set of $\mathrm{M}$ containing $i$ does not contain any other element in $\text{cl}(i)$.
The second and third statements follow directly from the definitions.
\end{proof}

\subsection{Weights}\label{sec:weights}

For any simplicial fan $\Sigma$, we write $\Sigma_k$ for the set of $k$-dimensional cones in $\Sigma$.
If $\tau$ is a codimension $1$ face of a cone $\sigma$, we write
\[
\mathbf{e}_{\sigma/\tau}\coloneq \text{the primitive generator of the unique ray in $\sigma$ that is not in $\tau$}.
\]
A \emph{$k$-dimensional balanced weight} on $\Sigma$ is a $\mathbb{Q}$-valued function 
$\omega$ on $\Sigma_k$ that satisfies the \emph{balancing condition}: For every $(k-1)$-dimensional cone $\tau$ in $\Sigma$,
\[
\text{$\sum_{\tau \subset \sigma} \omega(\sigma) \mathbf{e}_{\sigma/\tau}$ is contained in the subspace spanned by $\tau$,}
\]
where the sum is over all $k$-dimensional cones $\sigma$ containing $\tau$.
We write $\mathrm{MW}_k(\Sigma)$ for the group of $k$-dimensional balanced weights on $\Sigma$.

\begin{proposition}\label{PropositionUniqueBalancing}
The Bergman fan and the augmented Bergman fan of $\mathrm{M}$ have the following unique balancing property.
\begin{enumerate}[(1)]\itemsep 5pt
\item A $(d-1)$-dimensional weight on $\underline{\Pi}_\mathrm{M}$ is balanced if and only if it is constant.
\item A $d$-dimensional weight on $\Pi_\mathrm{M}$ is balanced if and only if it is constant.
\end{enumerate}
\end{proposition}

\begin{proof}
The first statement is \cite[Proposition 5.2]{AHK}.
We prove the second statement.

Let $\sigma_{I \le \mathscr{F}}$ be a codimension $1$ cone of $\Pi_\mathrm{M}$,
and let $F$ be the smallest flat in $\mathscr{F} \cup \{E\}$. 
We analyze the primitive generators of the rays in the star of the cone $\sigma_{I \le \mathscr{F}}$ in $\Pi_\mathrm{M}$.
Let $\text{cl}(I)$ be the closure of $I$ in $\mathrm{M}$.
There are two cases.

When the closure of $I$ is not $F$,
 the primitive ray generators in question are $-\mathbf{e}_{E \setminus \text{cl}(I)}$ and $\mathbf{e}_i$, for elements $i$ in $F$ not in the closure of $I$.
The primitive ray generators satisfy the relation
\[
-\mathbf{e}_{E \setminus \text{cl}(I)}+\sum_{i \in F \setminus \text{cl}(I)} \mathbf{e}_i=-\mathbf{e}_{E \setminus F},
\]
which is zero modulo the span of $\sigma_{I \le \mathscr{F}}$.
As the  $\mathbf{e}_i$'s   are independent modulo the span of $\sigma_{I \le \mathscr{F}}$,
any relation between the primitive generators must be a multiple of the displayed one.

When the closure of $I$ is $F$, the fact that $\sigma_{I \le \mathscr{F}}$ has codimension $1$ implies that
there is a unique integer $k$ with $\rk F < k < \rk \mathrm{M}$ such that $\mathscr{F}$ does not include a flat of rank $k$.
Let $F_\circ$ be the unique flat in $\mathscr{F}$ of rank $k-1$, and let $F^\circ$ be the unique flat in $\mathscr{F} \cup \{E\}$
of rank $k+1$.
The primitive ray generators in question are  $-\mathbf{e}_{E \setminus G}$ for the flats $G$ in $\mathscr{G}$, 
where  $\mathscr{G}$ is the set of flats of $\mathrm{M}$ covering $F_\circ$ and covered by $F^\circ$.
By the flat partition property of matroids \cite[Section 1.4]{Oxley}, the primitive ray generators satisfy the relation
\[
\sum_{G \in \mathscr{G}} -\mathbf{e}_{E \setminus G}=-(\ell-1)\mathbf{e}_{E \setminus F_\circ} -\mathbf{e}_{E \setminus F^\circ},
\]
which is zero modulo the span of $\sigma_{I \le \mathscr{F}}$.
Since any proper subset of the primitive generators $-\mathbf{e}_{E \setminus G}$  for $G$ in $\mathscr{G}$ is independent modulo the span of $\sigma_{I \le \mathscr{F}}$,
any relation between the primitive generators must be a multiple of the displayed one.

The local analysis above shows that any constant $d$-dimensional weight on $\Pi_\mathrm{M}$ is balanced.
Since $\Pi_\mathrm{M}$ is connected in codimension $1$ by Proposition \ref{PropositionConnected},
it also shows that any $d$-dimensional balanced weight  on $\Pi_\mathrm{M}$ must be constant.
\end{proof}

\subsection{Chow rings}\label{sec:rings}


Any unimodular fan $\Sigma$ in $\mathbb{R}^E$ defines a graded commutative algebra $\mathrm{CH}(\Sigma)$,
which is the Chow ring of the associated smooth toric variety $X_\Sigma$ over $\mathbb{C}$ with rational coefficients. 
Equivalently, $\mathrm{CH}(\Sigma)$ is the ring of continuous piecewise polynomial functions on $\Sigma$ with rational coefficients 
modulo the ideal generated by globally linear functions \cite[Section 3.1]{Brion}.
We write $\mathrm{CH}^k(\Sigma)$ for the Chow group of codimension $k$ cycles in $X_\Sigma$, so that
\[
\mathrm{CH}(\Sigma)=\bigoplus_{k} \mathrm{CH}^k(\Sigma).
\]
The group of $k$-dimensional balanced weights on $\Sigma$ is related to $\mathrm{CH}^k(\Sigma)$ by the isomorphism
\[
\mathrm{MW}_k(\Sigma) \longrightarrow \mathrm{Hom}_\mathbb{Q}(\mathrm{CH}^k(\Sigma),\mathbb{Q}), \quad \omega \longmapsto (x_\sigma \longmapsto \omega(\sigma)),
\]
where $x_\sigma$ is the class of the torus orbit closure in $X_\Sigma$ corresponding to a $k$-dimensional cone $\sigma$ in $\Sigma$.
See \cite[Section 5]{AHK} for a detailed discussion.
For general facts on toric varieties and Chow rings, and for any undefined terms, we refer to \cite{CLS} and \cite{Fulton}.

In Proposition~\ref{PropositionToricInterpretation} below, we show that the Chow ring of $\mathrm{M}$ coincides with $\CH(\underline{\Pi}_\mathrm{M})$
and that the augmented Chow ring of $\mathrm{M}$ coincides with $\CH(\Pi_\mathrm{M})$.

\begin{lemma}\label{PropositionIdentities}
The following identities hold in the augmented Chow ring $\mathrm{CH}(\mathrm{M})$.
\begin{enumerate}[(1)]\itemsep 5pt
\item For any element $i$ of $E$, we have  $y_i^2=0$.
\item For any two bases $I_1$ and $I_2$ of a flat $F$ of $\mathrm{M}$, we have $\prod_{i \in I_1} y_i = \prod_{i\in I_2} y_i$.
\item For any dependent set $J$ of $\mathrm{M}$, we have
$\prod_{j \in J} y_j=0$.
\end{enumerate}
\end{lemma}

\begin{proof}
The first identity is a straightforward consequence of the relations in $I_\mathrm{M}$ and $J_\mathrm{M}$:
\[
y_i^2=y_i \Big( \sum_{i \notin F} x_F\Big)=0. 
\]

For the second identity, 
we may assume that $I_1 \setminus I_2=\{i_1\}$ and $I_2 \setminus I_1 =\{i_2\}$,
by the basis exchange property of matroids.
Since a flat of $\mathrm{M}$ contains $I_1$ if and only if it contains $I_2$,
we have
\[
\Big(\sum_{i_1 \in G} x_G\Big)\prod_{i \in I_1 \cap I_2} y_i
=\Big(\sum_{I_1 \subseteq G} x_G\Big)\prod_{i \in I_1 \cap I_2} y_i
=
\Big(\sum_{I_2 \subseteq G} x_G\Big)\prod_{i \in I_1 \cap I_2} y_i
=\Big(\sum_{i_2 \in G} x_G\Big)\prod_{i \in I_1 \cap I_2} y_i.
\]
This immediately implies that we also have
$$\Big(\sum_{i_1 \notin G} x_G\Big)\prod_{i \in I_1 \cap I_2} y_i = \Big(\sum_{i_2 \notin G} x_G\Big)\prod_{i \in I_1 \cap I_2} y_i,$$
which tells us that
\[
\prod_{i \in I_1} y_i 
=
y_{i_1}\prod_{i \in I_1 \cap I_2} y_i
=
\Big(\sum_{i_1 \notin G} x_G\Big)\prod_{i \in I_1 \cap I_2} y_i
=
\Big(\sum_{i_2 \notin G} x_G\Big)\prod_{i \in I_1 \cap I_2} y_i
=
y_{i_2}\prod_{i \in I_1 \cap I_2} y_i
=
\prod_{i \in I_2} y_i. 
\]

For the third identity,
we may suppose that $J$ is a circuit, that is, a minimal dependent set.
Since $\mathrm{M}$ is a loopless matroid, we may choose distinct elements $j_1$ and $j_2$ from $J$.
Note that the independent sets $J \setminus j_1$ and $J \setminus j_2$ have the same closure because $J$ is a circuit.
Therefore, by the second identity, we have
\[
\prod_{j \in J \setminus j_1} y_j=
\prod_{j \in J \setminus j_2} y_j.
\]
Combining the above with the first identity, we get 
\[
\prod_{j \in J} y_j=
y_{j_1} \prod_{j \in J \setminus j_1} y_j=
y_{j_1} \prod_{j \in J \setminus j_2} y_j=
y_{j_1}^2 \prod_{j \in J \setminus \{j_1, j_2\}} y_j=0.\qedhere
\]
\end{proof}

By the second identity in Lemma \ref{PropositionIdentities}, we may define
 \[
 y_F \coloneq  \prod_{i \in I} y_i \ \ \text{in $\mathrm{CH}(\mathrm{M})$}
 \]
for any flat $F$ of $\mathrm{M}$ and any basis $I$ of $F$.
The element $y_E$ will play the role of the fundamental class for the augmented Chow ring of $\mathrm{M}$.

\begin{proposition}\label{PropositionToricInterpretation}
We have isomorphisms $$\underline{\mathrm{CH}}(\mathrm{M})\cong \mathrm{CH}(\underline{\Pi}_\mathrm{M})
\and
\mathrm{CH}(\mathrm{M})\cong \mathrm{CH}(\Pi_\mathrm{M}).$$
\end{proposition}


\begin{proof}
The first isomorphism is proved in \cite[Theorem 3]{FY}; see also \cite[Section 5.3]{AHK}.

Let $K_\mathrm{M}$ be the ideal of $S_\mathrm{M}$ generated by the monomials 
$\prod_{j \in J} y_j$ for every dependent set $J$ of $\mathrm{M}$.
The ring of continuous piecewise polynomial functions on $\Pi_\mathrm{M}$ is isomorphic to the Stanley--Reisner
ring of $\Delta_{\mathrm{M}}$, which is equal to $$S_\mathrm{M} / (J_\mathrm{M} + K_\mathrm{M}).$$
The ring $\mathrm{CH}(\Pi_\mathrm{M})$ is obtained from this ring by killing the linear forms that generate the ideal $I_\mathrm{M}$.
In other words, we have a surjective homomorphism
$$\mathrm{CH}(\mathrm{M}) \coloneq S_\mathrm{M} / (I_\mathrm{M} + J_\mathrm{M})
\longrightarrow
S_\mathrm{M} / (I_\mathrm{M} + J_\mathrm{M}  + K_\mathrm{M})
\cong \mathrm{CH}(\Pi_\mathrm{M}).$$
The fact that this is an isomorphism follows from the third part of Lemma \ref{PropositionIdentities}.
\end{proof}

\begin{remark}
By Proposition \ref{PropositionToricInterpretation},
the graded dimension of the Chow ring of the rank $d$ Boolean matroid  $\underline{\CH}(\mathrm{B})$
is given by the $h$-vector of the permutohedron in $\mathbb{R}^E$.
In other words, we have
\[
\dim \underline{\CH}^k(\mathrm{B})=\text{the Eulerian number $\eulerian{d}{k}$}.
\]
See \cite[Section 9.1]{Petersen} for more on permutohedra and Eulerian numbers.
\end{remark}

If $E$ is nonempty, we have the balanced weight
\[
1 \in \mathrm{MW}_{d-1}(\underline{\Pi}_\mathrm{M}) \cong \text{Hom}_\mathbb{Q}(\underline{\mathrm{CH}}^{d-1}(\mathrm{M}),\mathbb{Q}),
\]
which can be used to define a degree map on the Chow ring of $\mathrm{M}$.
Similarly, for any $E$, 
\[
1 \in \mathrm{MW}_{d}(\Pi_\mathrm{M}) \cong \text{Hom}_\mathbb{Q}(\mathrm{CH}^{d}(\mathrm{M}),\mathbb{Q})
\]
can be used to define a degree map on the augmented Chow ring of $\mathrm{M}$.

\begin{definition}\label{DefinitionDegreemap}
Consider the following \emph{degree maps} for the Chow ring and the augmented Chow ring of $\mathrm{M}$.
\begin{enumerate}[(1)]\itemsep 5pt
\item If $E$ is nonempty, the degree map for $\underline{\mathrm{CH}}(\mathrm{M})$ is the linear map
\[
\underline{\deg}_\mathrm{M}:  \underline{\mathrm{CH}}^{d-1}(\mathrm{M}) \longrightarrow \mathbb{Q}, \quad x_{\mathscr{F}} \longmapsto 1,
\]
where $x_\mathscr{F}$ is any monomial corresponding to a maximal cone $\underline{\sigma}_\mathscr{F}$ of $\underline{\Pi}_\mathrm{M}$.
\item For any $E$, the degree map for $\mathrm{CH}(\mathrm{M})$ is the linear map
\[
\deg_\mathrm{M}:  \mathrm{CH}^{d}(\mathrm{M}) \longrightarrow \mathbb{Q}, \quad x_{I \le \mathscr{F}} \longmapsto 1,
\]
where $x_{I\le \mathscr{F}}$ is any monomial corresponding to a maximal cone $\sigma_{I\le \mathscr{F}}$ of $\Pi_\mathrm{M}$.
\end{enumerate}
\end{definition}

By Proposition \ref{PropositionUniqueBalancing}, the degree maps are well-defined and are isomorphisms.
It follows that, for any two maximal cones $\underline{\sigma}_{\mathscr{F}_1}$ and $\underline{\sigma}_{\mathscr{F}_2}$ of the Bergman fan of $\mathrm{M}$, 
\[
x_{\mathscr{F}_1}=x_{\mathscr{F}_2} \ \ \text{in $\underline{\mathrm{CH}}^{d-1}(\mathrm{M})$}.
\]
Similarly, for any two maximal cones $\sigma_{I_1 \le \mathscr{F}_1}$ and $\sigma_{I_2 \le \mathscr{F}_2}$ of the augmented Bergman fan of $\mathrm{M}$,
\[
y_{F_1} x_{\mathscr{F}_1}=y_{F_2}x_{\mathscr{F}_2} \ \ \text{in $\mathrm{CH}^d(\mathrm{M})$},
\]
where $F_1$ is the closure of $I_1$ in $\mathrm{M}$ and $F_2$ is the closure of $I_2$ in $\mathrm{M}$.
Proposition \ref{PropositionToricInterpretation} shows that
\[
\underline{\mathrm{CH}}^k(\mathrm{M})=0 \ \ \text{for $k \ge d$} \ \  
\text{and} \ \ 
\mathrm{CH}^k(\mathrm{M})=0 \ \ \text{for $k > d$.}
\]

\begin{remark}\label{remark:wonderful}
Let $\mathbb{F}$ be a field,
and let $V$ be a $d$-dimensional linear subspace of $\mathbb{F}^E$.
We suppose that
the subspace $V$ is not contained in $\mathbb{F}^S \subseteq\mathbb{F}^E$
 for any proper subset $S$ of $E$.
Let $\mathrm{B}$ be the Boolean matroid on $E$,
and let  $\mathrm{M}$ be the loopless matroid on $E$ defined by
\[
\text{$S$ is an independent set of $\mathrm{M}$} \Longleftrightarrow
\text{the restriction to $V$ of the projection $\mathbb{F}^E \to \mathbb{F}^S$  is surjective}. 
\]
Let $\mathbb{P}(\mathbb{F}^E)$ be the projective space of lines in $\mathbb{F}^E$, 
and let $\underline{\mathbb{T}}_E$ be its open torus. 
For any  proper flat $F$ of $\mathrm{M}$, we write  $H_F$ for the projective subspace
\[
H_F \coloneq \big\{ p\in \mathbb{P}(V)\mid \text{$p_i = 0$ for all $i\in F$} \big\}.
\]  
The  \emph{wonderful variety} $\underline{X}_V$ is obtained from $\mathbb{P}(V)$ by first blowing up $H_F$ for every corank $1$ flat $F$, 
then blowing up the strict transforms of $H_F$ for every corank $2$ flat $F$, and so on.
Equivalently, 
\begin{align*}
\underline{X}_V&=\text{the closure of $\mathbb{P}(V) \cap \underline{\mathbb{T}}_E$ in the toric variety $\underline{X}_\mathrm{M}$ defined by $\underline{\Pi}_\mathrm{M}$}\\
&=\text{the closure of $\mathbb{P}(V) \cap \underline{\mathbb{T}}_E$ in the toric variety $\underline{X}_\mathrm{B}$ defined by $\underline{\Pi}_\mathrm{B}$}.
\end{align*}
When $E$ is nonempty, the inclusion $\underline{X}_V \subseteq \underline{X}_\mathrm{M}$ induces an isomorphism between their Chow rings,\footnote{In general, the inclusion $\underline{X}_V \subseteq \underline{X}_\mathrm{M}$ does not induce an isomorphism between their singular cohomology rings.}
and hence  the Chow ring of $\underline{X}_V$ is isomorphic to $\underline{\CH}(\mathrm{M})$ \cite[Corollary 2]{FY}.

Let $\mathbb{P}(\mathbb{F}^E \oplus\mathbb{F})$ be the projective completion of $\mathbb{F}^E$,
and let $\mathbb{T}_E$ be its  open torus. 
The projective completion  $\mathbb{P}(V \oplus\mathbb{F})$  contains
a copy of $\mathbb{P}(V)$
as the hyperplane at infinity, and it therefore contains a copy of $H_F$ for every nonempty proper flat $F$.
The \emph{augmented wonderful variety} $X_V$ is obtained from
$\mathbb{P}(V\oplus\mathbb{F}^1)$
by first blowing up $H_F$ for every corank $1$ flat $F$, 
then blowing up the strict transforms of $H_F$ for every corank $2$ flat $F$, and so on.
Equivalently,
\begin{align*}
X_V&=\text{the closure of $\mathbb{P}(V \oplus \mathbb{F}) \cap \mathbb{T}_E$ in the toric variety $X_\mathrm{M}$ defined by $\Pi_\mathrm{M}$}\\
&=\text{the closure of $\mathbb{P}(V \oplus \mathbb{F}) \cap \mathbb{T}_E$ in the toric variety $X_\mathrm{B}$ defined by $\Pi_\mathrm{B}$}.
\end{align*}
The inclusion $X_V \subseteq X_\mathrm{M}$ induces an isomorphism between their Chow rings, 
and hence the Chow ring of $X_V$ is isomorphic to $\CH(\mathrm{M})$.\footnote{This can be proved using the interpretation of $\mathrm{CH}(\mathrm{M})$ in 
the last sentence of Remark \ref{RemarkAlternative}.}
\end{remark}

\subsection{The graded M\"obius algebra}\label{sec:mobius}

For any nonnegative integer $k$, we define a vector space
\[
\mathrm{H}^k(\mathrm{M})\coloneq \bigoplus_{F\in \mathscr{L}^k(\mathrm{M})} \mathbb{Q} \hspace{0.5mm} y_F,
\]
where the direct sum is over the set $\mathscr{L}^k(\mathrm{M})$ of rank $k$ flats of $\mathrm{M}$.

\begin{definition}\label{DefinitionMobiusAlgebra}
The \emph{graded M\"obius algebra} of $\mathrm{M}$
is the graded vector space
\[
\mathrm{H}(\mathrm{M} )\coloneq \bigoplus_{k \ge 0} \mathrm{H}^k(\mathrm{M}).
\]
The multiplication in $\mathrm{H}(\mathrm{M})$ is defined by the rule
\[
y_{F_1}y_{F_2}=\begin{cases} y_{F_1 \lor F_2}& \text{if $\text{rk}_\mathrm{M}(F_1)+\text{rk}_\mathrm{M}(F_2) = \text{rk}_\mathrm{M}(F_1 \lor F_2)$,} \\
\hfil 0 & \text{if $\text{rk}_\mathrm{M}(F_1)+\text{rk}_\mathrm{M}(F_2) > \text{rk}_\mathrm{M}(F_1 \lor F_2)$,} \end{cases}
\]
where $\lor$ stands for the join operation in the lattice of flats $\mathscr{L}(\mathrm{M})$ of $\mathrm{M}$.
\end{definition}

Our double use of the symbol $y_F$ is justified by the following proposition.

\begin{proposition}\label{PropositionMobiusAlgebra}
The graded linear map
\[
\mathrm{H}(\mathrm{M} ) \longrightarrow \mathrm{CH}(\mathrm{M}), \quad y_F \longmapsto y_F
\]
is an injective homomorphism of graded algebras.
\end{proposition}

\begin{proof}
We first show that the linear map is injective.
It is enough to check that the subset
\[
\{y_F\}_{F \in \mathscr{L}^k(\mathrm{M})} \subseteq \mathrm{CH}^k(\mathrm{M})
\]
is linearly independent for every nonnegative integer $k<d$.
Suppose that
\[
\sum_{F \in \mathscr{L}^k(\mathrm{M})}  c_F y_F=0 \ \ \text{for some $c_F \in \mathbb{Q}$}.
\]
For any given rank $k$ flat $G$, we choose a saturated flag of proper flats $\mathscr{G}$ whose smallest member is $G$
and observe that
\[
c_G y_G x_\mathscr{G} = \Big(\sum_{F \in \mathscr{L}^k(\mathrm{M})}  c_F y_F\Big) x_\mathscr{G} =0.
\]
Since the degree of  $y_Gx_\mathscr{G} $ is $1$, this implies that $c_G$ must be zero.

We next check that the linear map is an algebra homomorphism using Lemma \ref{PropositionIdentities}.
Let $I_1$  be a basis of a flat $F_1$, and let $I_2$ be a basis of a flat $F_2$.
If the rank of $F_1 \lor F_2$ is the sum of the ranks of $F_1$ and $F_2$,
then $I_1$ and $I_2$ are disjoint and their union is a basis of $F_1 \lor F_2$.
Therefore, in the augmented Chow ring of $\mathrm{M}$,  
\[
y_{F_1}y_{F_2}=\prod_{i \in I_1} y_i \prod_{i \in I_2} y_i =\prod_{i \in I_1 \cup I_2} y_i=y_{F_1 \lor F_2}.
\]
If the rank of $F_1 \lor F_2$ is less than the sum of the ranks of $F_1$ and $F_2$,
then either $I_1$ and $I_2$ intersect or the union of $I_1$ and $I_2$ is dependent in $\mathrm{M}$.
Therefore, in the augmented Chow ring of $\mathrm{M}$,  
\[
y_{F_1}y_{F_2}=\prod_{i \in I_1} y_i \prod_{i \in I_2} y_i =0.
\qedhere
\]
\end{proof}

\begin{remark}
Consider the torus $\mathbb{T}_E$, the  toric variety $X_\mathrm{B}$, and 
the augmented wonderful variety $X_V$ in Remark \ref{remark:wonderful}.
The identity of $\mathbb{T}_E$ uniquely extends to a toric map 
\[
\mathrm{p}_\mathrm{B}: X_\mathrm{B} \longrightarrow (\mathbb{P}^1)^E.
\]
Let $\mathrm{p}_V$ be the restriction of $\mathrm{p}_\mathrm{B}$ to the augmented wonderful variety $X_V$.
If we identify  the Chow ring of $X_V$ with $\CH(\mathrm{M})$  as in Remark \ref{remark:wonderful},
 the image of the pullback $\mathrm{p}_V^*$ is
  the graded M\"obius algebra $\mathrm{H}(\mathrm{M}) \subseteq \CH(\mathrm{M})$.
  \end{remark}

\subsection{Pullback and pushforward maps}\label{sec:gysin}

Let $\Sigma$ be a unimodular fan, and let $\sigma$ be a $k$-dimensional cone in $\Sigma$.
The torus orbit closure in the smooth toric variety $X_\Sigma$ corresponding to $\sigma$ can be identified with
the toric variety of the fan $\text{star}_\sigma \Sigma$. 
Its class  in the Chow ring of $X_\Sigma$ is the monomial $x_\sigma$, which is 
the product of the divisor classes $x_\rho$ corresponding to the rays $\rho$ in $\sigma$.
The inclusion $\iota$ of the torus orbit closure in $X_\Sigma$ defines the \emph{pullback}  $\iota^*$ and the \emph{pushforward} $\iota_*$ between the Chow rings,
 whose composition
is multiplication by the monomial $x_\sigma$:
\[
\xymatrix{
\mathrm{CH}(\Sigma)  \ar[rr]^{x_\sigma} \ar[dr]_{\iota^*}&&\mathrm{CH}(\Sigma) 
\\
& \mathrm{CH}(\text{star}_\sigma\Sigma) \ar[ur]_{\iota_*}&
}
\]
The pullback $\iota^*$ is a surjective graded algebra homomorphism, while
the pushforward $\iota_*$ is a degree $k$ homomorphism of $\mathrm{CH}(\Sigma)$-modules.

We give an explicit description of the pullback $\iota^*$ and the pushforward $\iota_*$ when $\Sigma$ is the augmented Bergman fan $\Pi_\mathrm{M}$ and $\sigma$ is the ray $\rho_F$ of a proper flat $F$ of $\mathrm{M}$.
Recall from Proposition \ref{PropositionStars}  that the star of $\rho_F$ admits the decomposition
\[
 \text{star}_{\rho_F} \Pi_\mathrm{M} \cong \underline{\Pi}_{\mathrm{M}_{F}}\times \Pi_{\mathrm{M}^F}.
\]
Thus we may identify the Chow ring of the star of $\rho_F$ with $ \underline{\mathrm{CH}}(\mathrm{M}_F)\otimes \mathrm{CH}(\mathrm{M}^F)$.
We denote the  pullback to the tensor product by $\varphi_\mathrm{M}^F$ and the pushforward  from the tensor product by $\psi_\mathrm{M}^F$:
\[
\xymatrix{
\mathrm{CH}(\mathrm{M})  \ar[rr]^{x_F} \ar[dr]_{\varphi^F_\mathrm{M}}&&\mathrm{CH}(\mathrm{M}) 
\\
& \underline{\mathrm{CH}}(\mathrm{M}_F)\otimes \mathrm{CH}(\mathrm{M}^F)  \ar[ur]_{\psi^F_\mathrm{M}}&
}
\]
To  describe the pullback and the pushforward, we introduce Chow classes $\alpha_{\mathrm{M}}$, $\underline{\alpha}_{\mathrm{M}}$, and $\underline{\beta}_{\mathrm{M}}$. 
They are defined as the sums
\[
\alpha_{\mathrm{M}} \coloneq \sum_G x_G \in \mathrm{CH}^1(\mathrm{M}),
\]
where the sum is over all proper flats $G$ of $\mathrm{M}$;
\[
\underline{\alpha}_{\mathrm{M}}\coloneq  \sum_{i \in G} x_G \in \underline{\mathrm{CH}}^1(\mathrm{M}),
\]
where the sum is over all nonempty proper flats $G$ of $\mathrm{M}$ containing a given element $i$ in $E$; and
\[
\underline{\beta}_{\mathrm{M}} \coloneq \sum_{i \notin G} x_G  \in \underline{\mathrm{CH}}^1(\mathrm{M}),
\]
where the sum is over all nonempty proper flats $G$ of $\mathrm{M}$ not containing a given element $i$ in $E$.
The linear relations defining $\underline{\mathrm{CH}}(\mathrm{M})$ show that  $\underline{\alpha}_{\mathrm{M}}$ and $\underline{\beta}_{\mathrm{M}}$ do not depend on the choice of $i$. 

The following two propositions are straightforward.

\begin{proposition}\label{DefinitionXPullback}
The pullback 
$
\varphi^F_\mathrm{M}
$
is the unique graded algebra homomorphism 
\[
\mathrm{CH}(\mathrm{M}) \longrightarrow  \underline{\mathrm{CH}}(\mathrm{M}_F) \otimes \mathrm{CH}(\mathrm{M}^F) 
\]
 that satisfies the following properties:
\begin{enumerate}[$\bullet$]\itemsep 5pt
\item If $G$ is a proper flat of $\mathrm{M}$ incomparable to $F$, then $\varphi^F_\mathrm{M}(x_G)=0$.
\item If $G$ is a proper flat of $\mathrm{M}$ properly contained in $F$, then $\varphi^F_\mathrm{M}(x_G)=1 \otimes x_G$.
\item If $G$ is a proper flat of $\mathrm{M}$  properly containing $F$, then $\varphi^F_\mathrm{M}(x_G)=x_{G \setminus F} \otimes 1$.
\item If $i$ is an element of $F$, then $\varphi^F_\mathrm{M}(y_i)=1 \otimes y_i$.
\item If $i$ is an element of $E \setminus F$, then $\varphi^F_\mathrm{M}(y_i)=0$.
 \end{enumerate}
 The above five properties imply the following additional properties of $\varphi^F_\mathrm{M}$:
 \begin{enumerate}[$\bullet$]\itemsep 5pt
\item The equality $\varphi^F_\mathrm{M}(x_F)=-1\otimes\alpha_{\mathrm{M}^F}-\underline{\beta}_{\mathrm{M}_F} \otimes 1$ holds.
\item The equality $\varphi^F_\mathrm{M}(\alpha_\mathrm{M})=\underline{\alpha}_{\mathrm{M}_F} \otimes 1$ holds.
\end{enumerate}
\end{proposition}

\begin{proposition}\label{DefinitionXPushforward}
The pushforward $\psi^F_\mathrm{M}$ is the unique $\CH(\mathrm{M})$-module homomorphism\footnote{We make $\psi^F_{\mathrm{M}}$ into a $\mathrm{CH}(\mathrm{M})$-module homomorphism via the pullback $\varphi^F_{\mathrm{M}}$.}
\[
 \psi^F_\mathrm{M}: \underline{\mathrm{CH}}(\mathrm{M}_F)  \otimes \mathrm{CH}(\mathrm{M}^F) \longrightarrow \mathrm{CH}(\mathrm{M})
 \]
that satisfies,
for any collection $\mathscr{S}'$ of proper flats of $\mathrm{M}$ strictly containing $F$
and any collection $\mathscr{S}''$ of proper flats of $\mathrm{M}$ strictly contained in $F$,
\[
\psi^F_\mathrm{M}\Bigg( \prod_{F' \in \mathscr{S}'} x_{F' \setminus F}  \otimes \prod_{F'' \in \mathscr{S}''} x_{F''}\Bigg)=x_F \prod_{F' \in \mathscr{S}'} x_{F'} \prod_{F'' \in \mathscr{S}''} x_{F''}.
\]
The composition  $\psi_\mathrm{M}^F \circ \varphi_\mathrm{M}^F$ is multiplication by the element $x_F$,
and  the composition  $\varphi_\mathrm{M}^F \circ \psi_\mathrm{M}^F$ is multiplication by the element $\varphi_\mathrm{M}^F(x_F)$.
\end{proposition}

Proposition \ref{DefinitionXPushforward} shows that the pushforward $\psi^F_\mathrm{M}$ commutes with the degree maps: 
\[
\underline{\deg}_{\mathrm{M}_F} \otimes \deg_{\mathrm{M}^F} = \deg_\mathrm{M} \circ\ \psi_\mathrm{M}^F.
\]



\begin{proposition}\label{PropositionPushforwardI}
If   $\underline{\mathrm{CH}}(\mathrm{M}_F)$ and $\mathrm{CH}(\mathrm{M}^F)$ satisfy the Poincar\'e duality part of Theorem \ref{TheoremChowKahlerPackage},
then 
 $\psi_\mathrm{M}^F$ is injective.
\end{proposition}

In other words, assuming Poincar\'e duality for the Chow rings, the graded $\mathrm{CH}(\mathrm{M})$-module $ \underline{\mathrm{CH}}(\mathrm{M}_F) \otimes \mathrm{CH}(\mathrm{M}^F)[-1]$ 
is isomorphic to  the principal ideal of $x_F$ in $\mathrm{CH}(\mathrm{M})$.\footnote{For a graded vector space $V$, we write $V[m]$ for the graded vector space whose degree $k$ piece is equal to $V^{k+m}$.}
In particular,
\[
\underline{\mathrm{CH}}(\mathrm{M})[-1] \cong \text{ideal}(x_\varnothing) \subseteq \mathrm{CH}(\mathrm{M}).
\]

\begin{proof}
%
We will use the symbol $\deg_F$ to denote the degree function $\underline{\deg}_{\mathrm{M}_F} \otimes \deg_{\mathrm{M}^F}$.
For contradiction, suppose that $\psi^F_{\mathrm{M}}(\eta)=0$ for $\eta \neq 0$.
By the two Poincar\'e duality statements in Theorem \ref{TheoremChowKahlerPackage},
there is an element $\nu$  such that $\deg_F(\nu \eta)=1$.
By surjectivity of the pullback $\varphi^F_\mathrm{M}$,
there is an element $\mu$ such that $\nu=\varphi^F_{\mathrm{M}}(\mu)$.
Since $\psi^F_{\mathrm{M}}$ is a $\mathrm{CH}(\mathrm{M})$-module homomorphism that commutes with the degree maps,  we have
\[
1=\deg_F(\nu\eta)= \deg_{\mathrm{M}}( \psi^F_{\mathrm{M}} (\nu \eta))= \deg_{\mathrm{M}}( \psi^F_{\mathrm{M}} ( \varphi^F_{\mathrm{M}}(\mu) \eta))
= \deg_{\mathrm{M}}(\mu  \psi^F_{\mathrm{M}} ( \eta))
=\deg_{\mathrm{M}}(0)=0,
\]
which is a contradiction.
\end{proof}

We next give an explicit description of the pullback $\iota^*$ and the pushforward $\iota_*$ when $\Sigma$ is the  Bergman fan $\underline{\Pi}_\mathrm{M}$ and $\sigma$ is the ray $\underline{\rho}_F$ of  a nonempty proper flat $F$ of $\mathrm{M}$.
Recall from Proposition \ref{PropositionStars} that the star of $\underline{\rho}_F$ admits the decomposition
\[
 \text{star}_{\underline{\rho}_F} \underline{\Pi}_\mathrm{M} \cong \underline{\Pi}_{\mathrm{M}_{F}}\times \underline{\Pi}_{\mathrm{M}^F}.
\]
Thus  we may identify the Chow ring of the star of $\underline{\rho}_F$ with $ \underline{\mathrm{CH}}(\mathrm{M}_F)\otimes \underline{\mathrm{CH}}(\mathrm{M}^F)$. 
We denote the pullback to the tensor product by $\underline{\varphi}_\mathrm{M}^F$ and the pushforward  from the tensor product by $\underline{\psi}_\mathrm{M}^F$:
\[
\xymatrix{
\underline{\mathrm{CH}}(\mathrm{M})  \ar[rr]^{x_F} \ar[dr]_{\underline{\varphi}^F_\mathrm{M}}&&\underline{\mathrm{CH}}(\mathrm{M}) 
\\
& \underline{\mathrm{CH}}(\mathrm{M}_F)\otimes \underline{\mathrm{CH}}(\mathrm{M}^F)  \ar[ur]_{\underline{\psi}^F_\mathrm{M}}&
}
\]
The following analogues of Propositions \ref{DefinitionXPullback} and \ref{DefinitionXPushforward} are straightforward.

\begin{proposition}\label{DefinitionUnderlinedPull}
The pullback $\underline{\varphi}^F_\mathrm{M}$ is the unique graded algebra homomorphism
\[
 \underline{\mathrm{CH}}(\mathrm{M}) \longrightarrow  \underline{\mathrm{CH}}(\mathrm{M}_F) \otimes \underline{\mathrm{CH}}(\mathrm{M}^F) 
\]
that satisfies the following properties:
\begin{enumerate}[$\bullet$]\itemsep 5pt
\item If $G$ is a nonempty proper flat of $\mathrm{M}$ incomparable to $F$, then $\underline{\varphi}^F_\mathrm{M}(x_G)=0$.
\item If $G$ is a nonempty proper flat of $\mathrm{M}$ properly contained in $F$, then $\underline{\varphi}^F_\mathrm{M}(x_G)=1 \otimes x_G$.
\item If $G$ is a nonempty proper flat of $\mathrm{M}$ properly containing $F$, then $\underline{\varphi}^F_\mathrm{M}(x_G)=x_{G \setminus F} \otimes 1$.
\end{enumerate}
 The above three properties imply the following additional properties of $\underline{\varphi}^F_\mathrm{M}$:
\begin{enumerate}[$\bullet$]\itemsep 5pt
\item The equality $\underline{\varphi}^F_\mathrm{M}(x_F)=-1\otimes\underline{\alpha}_{\mathrm{M}^F}-\underline{\beta}_{\mathrm{M}_F} \otimes 1$ holds.
\item The equality $\underline{\varphi}^F_\mathrm{M}(\underline{\alpha}_\mathrm{M})= \underline{\alpha}_{\mathrm{M}_F} \otimes 1$ holds. 
\item The equality $\underline{\varphi}^F_\mathrm{M}(\underline{\beta}_\mathrm{M})=1 \otimes \underline{\beta}_{\mathrm{M}^F}$ holds. \end{enumerate}
\end{proposition}

\begin{proposition}\label{DefinitionUnderlinedPush}
The pushforward $ \underline{\psi}^F_\mathrm{M}$ is the unique  $ \underline{\mathrm{CH}}(\mathrm{M})$-module homomorphism
\[
\underline{\mathrm{CH}}(\mathrm{M}_F)  \otimes \underline{\mathrm{CH}}(\mathrm{M}^F) \longrightarrow \underline{\mathrm{CH}}(\mathrm{M})
 \]
 that satisfies,
for any collection $\mathscr{S}'$ of proper flats of $\mathrm{M}$ strictly containing $F$
and any collection $\mathscr{S}''$ of nonempty proper flats of $\mathrm{M}$ strictly contained in $F$,
\[
\underline{\psi}^F_\mathrm{M}\Bigg( \prod_{F' \in \mathscr{S}'} x_{F' \setminus F}  \otimes \prod_{F'' \in \mathscr{S}''} x_{F''}\Bigg)=x_F \prod_{F' \in \mathscr{S}'} x_{F'} \prod_{F'' \in \mathscr{S}''} x_{F''}.
\]
The composition  $\underline{\psi}_\mathrm{M}^F \circ \underline{\varphi}_\mathrm{M}^F$ is multiplication by the element $x_F$,
and  the composition  $\underline{\varphi}_\mathrm{M}^F \circ \underline{\psi}_\mathrm{M}^F$ is multiplication by the element $\underline{\varphi}_\mathrm{M}^F(x_F)$.
\end{proposition}

Proposition \ref{DefinitionUnderlinedPush} shows that
the pushforward $\underline{\psi}^F_\mathrm{M}$ commutes with the degree maps: 
\[
\underline{\deg}_{\mathrm{M}_F} \otimes \underline{\deg}_{\mathrm{M}^F} = \underline{\deg}_\mathrm{M} \circ\ \underline{\psi}_\mathrm{M}^F.
\]


\begin{proposition}\label{upsi injective}
If   $\underline{\mathrm{CH}}(\mathrm{M}_F)$ and $\underline{\mathrm{CH}}(\mathrm{M}^F)$ satisfy the Poincar\'e duality part of Theorem \ref{TheoremChowKahlerPackage},
then 
$\underline{\psi}^F_\mathrm{M}$ is injective.
\end{proposition}

In other words, assuming Poincar\'e duality for the Chow rings,
 the graded $\underline{\mathrm{CH}}(\mathrm{M})$-module $ \underline{\mathrm{CH}}(\mathrm{M}_F) \otimes \underline{\mathrm{CH}}(\mathrm{M}^F)[-1]$ 
is isomorphic to  the principal ideal of $x_F$ in $\underline{\mathrm{CH}}(\mathrm{M})$.

\begin{proof}
The proof is essentially identical to that of Proposition \ref{PropositionPushforwardI}.
\end{proof}

Last, we give an explicit description of the pullback $\iota^*$ and the pushforward $\iota_*$ when $\Sigma$ is the  augmented Bergman fan $\Pi_\mathrm{M}$ and $\sigma$ is the cone $\sigma_I$ of  a  independent set  $I$ of $\mathrm{M}$.
By Proposition \ref{PropositionStars}, we have
\[
 \text{star}_{\sigma_I} \Pi_\mathrm{M} \cong \Pi_{\mathrm{M}_F},
\]
where $F$ is the closure of $I$ in $\mathrm{M}$.
Thus we may identify the Chow ring of the star  of $\sigma_I$ with  $\mathrm{CH}(\mathrm{M}_F)$.
We denote the corresponding pullback  by $\varphi^\mathrm{M}_F$ and the pushforward  by $\psi^\mathrm{M}_F$:
\[
\xymatrix{
\mathrm{CH}(\mathrm{M})  \ar[rr]^{y_F} \ar[dr]_{\varphi_F^\mathrm{M}}&&\mathrm{CH}(\mathrm{M}) 
\\
& \mathrm{CH}(\mathrm{M}_F)  \ar[ur]_{\psi_F^\mathrm{M}}&
}
\]
Note that the pullback  and the pushforward only depend on $F$ and not on $I$.

The following analogues of Propositions  \ref{DefinitionXPullback} and \ref{DefinitionXPushforward} are straightforward.

\begin{proposition}\label{DefinitionYPull}
The pullback $\varphi_F^\mathrm{M}$ is the unique graded algebra homomorphism
\[
 \mathrm{CH}(\mathrm{M}) \longrightarrow  \mathrm{CH}(\mathrm{M}_F) 
\]
that satisfies the following properties:
\begin{enumerate}[$\bullet$]\itemsep 5pt
\item If $G$ is a proper flat of $\mathrm{M}$ that contains $F$, then $\varphi_F^{\mathrm{M}}(x_G)=x_{G\setminus F}$.
\item If $G$ is a proper flat of $\mathrm{M}$ that does not contain $F$, then $\varphi_F^{\mathrm{M}}(x_G)=0$.
\end{enumerate}
 The above two properties imply the following additional properties of $\varphi_F^\mathrm{M}$:
 \begin{enumerate}[$\bullet$]\itemsep 5pt
\item If $i$ is an element of $F$, then $\varphi_F^{\mathrm{M}}(y_i)=0$.
\item If $i$ is an element of $E \setminus F$, then $\varphi_F^{\mathrm{M}}(y_i)=y_i$.
\item The equality $\varphi_F^\mathrm{M}(\alpha_\mathrm{M})=\alpha_{\mathrm{M}_F}$ holds.
\end{enumerate}
\end{proposition}

\begin{proposition}\label{DefinitionYPush}
The pushforward  $\psi_F^\mathrm{M}$ is the unique $\mathrm{CH}(\mathrm{M})$-module homomorphism
\[
 \mathrm{CH}(\mathrm{M}_F) \longrightarrow \mathrm{CH}(\mathrm{M})
\]
that satisfies,
for any collection $\mathscr{S}'$ of proper flats of $\mathrm{M}$  containing $F$,
\[
\psi_F^\mathrm{M}\Bigg( \prod_{F' \in \mathscr{S}'} x_{F' \setminus F}\Bigg)=y_F  \prod_{F' \in \mathscr{S}'} x_{F'}.
\]
The composition  $\psi^\mathrm{M}_F \circ \varphi^\mathrm{M}_F$ is multiplication by the element $y_F$,
and  the composition  $\varphi^\mathrm{M}_F \circ \psi^\mathrm{M}_F$ is zero.
\end{proposition}

Proposition \ref{DefinitionYPush} shows that
 the pushforward $\psi_F^\mathrm{M}$ commutes with the degree maps: 
 \[
 \deg_{\mathrm{M}_F}=\deg_\mathrm{M} \circ\ \psi_F^\mathrm{M}.
 \]


\begin{proposition}\label{psi-injective}
If  $\mathrm{CH}(\mathrm{M}_F)$ satisfies the Poincar\'e duality part of Theorem \ref{TheoremChowKahlerPackage},
then $\psi_F^\mathrm{M}$ is injective.
\end{proposition}

In other words, assuming Poincar\'e duality for the Chow rings,
 the graded $\mathrm{CH}(\mathrm{M})$-module $ \mathrm{CH}(\mathrm{M}_F) [-\text{rk}_\mathrm{M}(F)]$ 
is isomorphic to  the principal ideal of $y_F$ in $\mathrm{CH}(\mathrm{M})$.

\begin{proof}
The proof is essentially identical to that of Proposition \ref{PropositionPushforwardI}.
\end{proof}

The basic properties of the pullback and the pushforward maps can be used to describe the fundamental classes
of $\underline{\CH}(\mathrm{M})$ and $\CH(\mathrm{M})$ in terms of $\underline{\alpha}_\mathrm{M}$ and $\alpha_\mathrm{M}$.

\begin{proposition}\label{PropositionAlphaDegree}
The degree of $\underline{\alpha}^{d-1}_\mathrm{M}$ is $1$, and the degree of $\alpha^d_\mathrm{M}$ is $1$.
\end{proposition}

\begin{proof}
We prove  the first statement by induction on  $d \ge 1$.
Note that, for any nonempty proper flat $F$ of rank $k$, we have
\[
x_F\hspace{0.3mm}  \underline{\alpha}_\mathrm{M}^{d-k}= \underline{\psi}^F_\mathrm{M}  \big( \underline{\varphi}^F_\mathrm{M} ( \underline{\alpha}_\mathrm{M}^{d-k} )\big)
= \underline{\psi}^F_\mathrm{M} \big(  \underline{\alpha}_\mathrm{M_F}^{d-k} \otimes 1 \big)=0,
\]
since $\underline{\mathrm{\CH}}^{d-k}(\mathrm{M}_F) = 0$.  Therefore, for any proper flat $a$ of rank $1$ and any element $i$ in $a$, we have
\[
\underline{\alpha}_\mathrm{M}^{d-1}= \Big(\sum_{i \in F} x_F\Big)\underline{\alpha}_\mathrm{M}^{d-2} =x_a\hspace{0.3mm}  \underline{\alpha}_\mathrm{M}^{d-2}.
\]
Now, using the induction hypothesis applied to the matroid $\mathrm{M}_a$ of rank $d-1$, we get
\[
\underline{\alpha}_\mathrm{M}^{d-1}=x_a\hspace{0.3mm}  \underline{\alpha}_\mathrm{M}^{d-2}
= \underline{\psi}^a_\mathrm{M}  \big( \underline{\varphi}^a_\mathrm{M} ( \underline{\alpha}_\mathrm{M}^{d-2} )\big)
= \underline{\psi}^a_\mathrm{M} \big(  \underline{\alpha}_{\mathrm{M}_a}^{d-2} \otimes 1 \big)
=x_\mathscr{F},
\]
where $\mathscr{F}$ is any maximal flag of nonempty proper flats of $\mathrm{M}$ that starts from $a$.

For the second statement,
note that, for any proper flat $F$ of rank $k$, 
\[
x_F\hspace{0.3mm}  \alpha_\mathrm{M}^{d-k}=\psi^F_\mathrm{M}  \big(\varphi^F_\mathrm{M} ( \alpha_\mathrm{M}^{d-k} )\big)
=\psi^F_\mathrm{M} \big(  \underline{\alpha}_\mathrm{M_F}^{d-k} \otimes 1 \big)=0.
\]
Using the first statement, we get the conclusion from the identity
\[
\alpha_\mathrm{M}^d=
 \Big(  \sum_F x_F \Big) \alpha_\mathrm{M}^{d-1}=x_\varnothing \hspace{0.3mm} \alpha_\mathrm{M}^{d-1}=
 \psi^\varnothing_\mathrm{M} \big( \varphi^\varnothing_\mathrm{M} ( \alpha_\mathrm{M}^{d-1} )\big)
=\psi^\varnothing_\mathrm{M} \big(\underline{\alpha}_\mathrm{M}^{d-1}\big). \qedhere
\]
\end{proof}

More generally, the degree of $\underline{\alpha}_\mathrm{M}^{d-k} \underline{\beta}_\mathrm{M}^k$ is the $k$-th coefficient of the reduced characteristic polynomial of $\mathrm{M}$ \cite[Proposition 9.5]{AHK}.

\begin{remark}
In the setting of Remark \ref{remark:wonderful},
the element $\alpha_\mathrm{M}$, viewed as an element of the Chow ring of the augmented wonderful variety $X_V$,
is the class of the pullback  of the hyperplane  $\mathbb{P}(V) \subseteq \mathbb{P}(V \oplus \mathbb{F})$.
\end{remark}

\section{Proofs of the semi-small decompositions and the Poincar\'e duality theorems}\label{Section3}

In this section, we prove Theorems \ref{TheoremUnderlinedDecomposition} and \ref{TheoremDecomposition} together with the two Poincar\'e duality statements in Theorem \ref{TheoremChowKahlerPackage}.
For an element $i$ of $E$, we write  $\pi_i$ and $\underline{\pi}_i$ for the coordinate projections
\[
\pi_i: \mathbb{R}^E \longrightarrow \mathbb{R}^{E \setminus i} \ \ \text{and} \ \ 
\underline{\pi}_i: \mathbb{R}^E / \langle \mathbf{e}_E \rangle\longrightarrow \mathbb{R}^{E \setminus i}/\langle \mathbf{e}_{E\setminus i} \rangle.
\]
Note that $\pi_i(\rho_i)=0$ and $\underline{\pi}_i(\underline{\rho}_{\{i\}})=0$.
In addition, $\pi_i(\rho_S)=\rho_{S \setminus i}$
and $\underline{\pi}_i(\underline{\rho}_S)=\underline{\rho}_{S \setminus i}$ for $S \subseteq E$.

\begin{proposition}\label{PropositionConeToCone}
Let $\mathrm{M}$ be a loopless matroid on $E$, and let $i$ be an element of $E$.
\begin{enumerate}[(1)]\itemsep 5pt
\item The projection $\pi_i$ maps any cone of   $\Pi_\mathrm{M}$ onto a cone of  $\Pi_{\mathrm{M} \setminus i}$.
\item  The projection $\underline{\pi}_i$ maps any cone of  $\underline{\Pi}_\mathrm{M}$ onto a cone of  $\underline{\Pi}_{\mathrm{M} \setminus i}$. 
\end{enumerate}
\end{proposition}

Recall that a linear map defines a morphism of fans $\Sigma_1 \to \Sigma_2$ if it maps any cone of $\Sigma_1$ into a cone of $\Sigma_2$ \cite[Chapter 3]{CLS}.
Thus the above proposition is  stronger than the statement that 
$\pi_i$ and $\underline{\pi}_i$ induce morphisms of fans. 

\begin{proof}
The projection $\pi_i$ maps $\sigma_{I \le \mathscr{F}}$ onto $\sigma_{I \setminus i \le \mathscr{F} \setminus i}$, where $\mathscr{F} \setminus i$ is the flag of flats of $\mathrm{M}\setminus i$
obtained by removing $i$ from the members of $\mathscr{F}$. Similarly, $\underline{\pi}_i$ maps $\underline{\sigma}_{\mathscr{F}}$  onto  $\underline{\sigma}_{\mathscr{F}\setminus i}$. 
\end{proof}
\noindent


By Proposition \ref{PropositionConeToCone},
the projection $\pi_i$ defines a map from the toric variety $X_\mathrm{M}$ of $\Pi_{\mathrm{M}}$ to the toric variety $X_{\mathrm{M} \setminus i}$ of $\Pi_{ {\mathrm{M} \setminus i}}$,
and hence the pullback homomorphism $ \mathrm{CH}(\mathrm{M}\setminus i)  \to \mathrm{CH}(\mathrm{M})$.
Explicitly, the pullback is the graded algebra homomorphism
\[
\theta_i=\theta^\mathrm{M}_i : \mathrm{CH}(\mathrm{M}\setminus i) \longrightarrow \mathrm{CH}(\mathrm{M}), \qquad x_F\longmapsto x_F + x_{F\cup i},
\]
where a variable in the target is set to zero if its label is not a flat of $\mathrm{M}$.
Similarly,  $\underline{\pi}_i$ defines a map from the toric variety $\underline{X}_\mathrm{M}$ of $\underline{\Pi}_{\mathrm{M}}$ to the toric variety $\underline{X}_{\mathrm{M} \setminus i}$ of $\underline{\Pi}_{ {\mathrm{M} \setminus i}}$,
and hence an algebra homomorphism
\[
\underline{\theta}_i =\underline{\theta}^\mathrm{M}_i : \underline{\mathrm{CH}}(\mathrm{M}\setminus i) \longrightarrow \underline{\mathrm{CH}}(\mathrm{M}), \qquad x_F\longmapsto x_F + x_{F\cup i},
\]
where a variable in the target is set to zero if its label is not a flat of $\mathrm{M}$.

\begin{remark}
We use the notations introduced in Remark \ref{remark:wonderful}.
Let $V \setminus i$ be the image of $V$ under the $i$-th projection
$\mathbb{F}^E \to \mathbb{F}^{E \setminus i}$.
We have the commutative diagrams of wonderful varieties and their Chow rings
\[
\begin{tikzcd}[column sep=10mm, row sep=10mm]
\underline{X}_\mathrm{B}  \arrow[d, two heads, "\underline{p}_i^\mathrm{B}"'] & \underline{X}_V \arrow[d, two heads, "\underline{p}_i^V"]  \arrow[l, hook']\\
\underline{X}_{\mathrm{B} \setminus i} & \underline{X}_{V \setminus i}, \arrow[l, hook']
\end{tikzcd}
\qquad  \qquad
\begin{tikzcd}[column sep=10mm, row sep=10mm]
\underline{\CH}(\mathrm{B}) \arrow[r, two heads]& \underline{\CH}(\mathrm{M})  \\
\underline{\CH}({\mathrm{B} \setminus i}) \arrow[u, hook]\arrow[r, two heads]& \underline{\CH}(\mathrm{M} \setminus i). \arrow[u, hook]
\end{tikzcd}
\]
The map  $\underline{p}_i^V$ is birational if and only if $i$ is not a coloop of $\mathrm{M}$.
By Proposition \ref{PropositionConeToCone}, the fibers of $\underline{p}_i^\mathrm{B}$ are at most one-dimensional,
and hence the fibers of $\underline{p}_i^V$ are at most one-dimensional.
It follows that  $\underline{p}_i^V$ is semi-small in the sense of Goresky--MacPherson when $i$ is not a coloop of $\mathrm{M}$.

Similarly, we have the  diagrams of augmented wonderful varieties and their Chow rings
\[
\begin{tikzcd}[column sep=10mm, row sep=10mm]
X_\mathrm{B}  \arrow[d, two heads, "p_i^\mathrm{B}"'] & X_V \arrow[d, two heads, "p_i^V"]  \arrow[l, hook']\\
X_{\mathrm{B} \setminus i} & X_{V \setminus i}, \arrow[l, hook']
\end{tikzcd}
\qquad  \qquad
\begin{tikzcd}[column sep=10mm, row sep=10mm]
\CH(\mathrm{B}) \arrow[r, two heads]& \CH(\mathrm{M})  \\
\CH({\mathrm{B} \setminus i}) \arrow[u, hook]\arrow[r, two heads]& \CH(\mathrm{M} \setminus i). \arrow[u, hook]
\end{tikzcd}
\]
The map  $p_i^V$ is birational if and only if $i$ is not a coloop of $\mathrm{M}$.
By Proposition \ref{PropositionConeToCone}, the fibers of  $p_i^\mathrm{B}$ are at most one-dimensional,
and hence 
$p_i^V$ is semi-small when $i$ is not a coloop of $\mathrm{M}$.

Numerically,
the semi-smallness of $\underline{p}_i^V$  
 is reflected in the identity
\[
\dim  x_{F\cup i}\underline{\mathrm{CH}}^{k-1}_{(i)}  = \dim  x_{F\cup i}\underline{\mathrm{CH}}_{(i)}^{d-k-2}.
\]
Similarly,
the semi-smallness of $p_i^V$ 
is reflected in the identity\footnote{The displayed identities follow from Proposition \ref{lem:top degree vanishing} and the Poincar\'e duality parts of Theorem \ref{TheoremChowKahlerPackage}.}
\[
\dim  x_{F\cup i}\mathrm{CH}^{k-1}_{(i)}  = \dim  x_{F\cup i}\mathrm{CH}_{(i)}^{d-k-1}.
\]
For a detailed discussion of semi-small maps in the context of Hodge theory and the decomposition theorem, see \cite{dCM2} and \cite{dCM1}.
\end{remark}

The element $i$ is said to be a \emph{coloop} of $\mathrm{M}$ if the ranks of $\mathrm{M}$ and $\mathrm{M} \setminus i$ are not equal.
We show that the pullbacks $\theta_i$ and $\underline{\theta}_i$ are compatible with the degree maps of $\mathrm{M}$ and $\mathrm{M} \setminus i$.

\begin{lemma}\label{lemma_PoincareCompatible}
Suppose that $E\setminus i$ is nonempty.
\begin{enumerate}[(1)]\itemsep 5pt
\item If $i$ is not a coloop of $\mathrm{M}$, then $\theta_i$ commutes with the degree maps:
\[
\deg_{\mathrm{M}\setminus i}=\deg_{\mathrm{M}} \circ\ \theta_i. 
\]
\item If $i$ is not a coloop of $\mathrm{M}$, then $\underline{\theta}_i$ commutes with the degree maps:
\[
\underline{\deg}_{\mathrm{M}\setminus i}= \underline{\deg}_{\mathrm{M}} \circ\ \underline{\theta}_i.
\]
\item If $i$ is a coloop of $\mathrm{M}$, we have
\[
\deg_{\mathrm{M}\setminus i}
=\deg_{\mathrm{M}} \circ\ x_{E \setminus i} \circ\ \theta_i 
=\deg_{\mathrm{M}} \circ\ \alpha_{\mathrm{M}} \circ\ \theta_i, 
\]
where the middle maps are multiplications by the elements $x_{E \setminus i}$ and $\alpha_\mathrm{M}$.
\item If $i$ is a coloop of $\mathrm{M}$, we have
\[
\underline{\deg}_{\mathrm{M}\setminus i}
=\underline{\deg}_{\mathrm{M}} \circ\ x_{E \setminus i} \circ\ \underline{\theta}_i
=\underline{\deg}_{\mathrm{M}} \circ\ \underline{\alpha}_\mathrm{M} \circ\ \underline{\theta}_i,
\]
where the middle maps are multiplications by the elements $x_{E \setminus i}$ and $\underline{\alpha}_\mathrm{M}$.
\end{enumerate}
\end{lemma}

\begin{proof}
If $i$ is not a coloop of $\mathrm{M}$, we may choose a basis $B$ of $\mathrm{M} \setminus i$ that is also a basis of $\mathrm{M}$. 
We have
\[
 \mathrm{CH}^d(\mathrm{M}\setminus i)=\text{span}(y_B) \ \ \text{and} \ \  \mathrm{CH}^d(\mathrm{M})=\text{span}(y_B).
 \]
Since $\theta_i(y_j)=y_j$ for all $j$, the first identity follows.
Similarly, by  Proposition \ref{PropositionAlphaDegree}, 
\[
\underline{\mathrm{CH}}^{d-1}(\mathrm{M}\setminus i)=\text{span}(\underline\alpha_{\mathrm{M}\setminus i}^{d-1}) \ \ \text{and} \ \   \underline{\mathrm{CH}}^{d-1}(\mathrm{M})=\text{span}(\underline\alpha_{\mathrm{M}}^{d-1}).
\]
Since  $\underline{\theta}_i(\underline\alpha_{\mathrm{M}\setminus i}) =\underline\alpha_{\mathrm{M}}$ when $i$ is not a coloop, the second identity follows.

Suppose now that $i$ is a coloop of $\mathrm{M}$. 
In this case, 
$\mathrm{M} \setminus i=\mathrm{M}^{E \setminus i}$,
and hence
\[
\varphi^{E \setminus i}_\mathrm{M} \circ \ \theta_i=\text{identity of $\CH(\mathrm{M} \setminus i)$}
\ \ \text{and} \ \ 
\underline{\varphi}^{E \setminus i}_\mathrm{M} \circ \ \underline{\theta}_i=\text{identity of $\underline{\CH}(\mathrm{M} \setminus i)$}.
\]
Using the compatibility of the pushforward $\psi^{E\setminus i}_\mathrm{M}$ with the degree maps,  we have
\[
\deg_{\mathrm{M} \setminus i}
=\deg_\mathrm{M} \circ\ \psi^{E \setminus i}_\mathrm{M} 
=\deg_\mathrm{M} \circ\ \psi^{E \setminus i}_\mathrm{M} \circ \ \varphi^{E \setminus i}_\mathrm{M} \circ \ \theta_i
=\deg_{\mathrm{M}} \circ\ x_{E \setminus i} \circ\ \theta_i.
\]
Since  $\theta_i(\alpha_{\mathrm{M} \setminus i})=\alpha_\mathrm{M}-x_{E \setminus i}$  when $i$ is a coloop of $\mathrm{M}$, the above implies
\[
\deg_{\mathrm{M} \setminus i}
=\deg_{\mathrm{M}} \circ\ x_{E \setminus i} \circ\ \theta_i
=\deg_{\mathrm{M}} \circ\ \big(\alpha_\mathrm{M} -\theta_i (\alpha_{\mathrm{M} \setminus i})\big) \circ\ \theta_i
=\deg_{\mathrm{M}} \circ\ \alpha_\mathrm{M} \circ\ \theta_i,
\]
The identities for $\underline{\deg}_{\mathrm{M} \setminus i}$ can be obtained in a similar way. 
\end{proof}


\begin{proposition}\label{DeletionInjection}
If $\mathrm{CH}(\mathrm{M} \setminus i)$ satisfies the Poincar\'e duality part of Theorem \ref{TheoremChowKahlerPackage}, then $\theta_i$ is injective.  Also, if $\underline{\mathrm{CH}}(\mathrm{M} \setminus i)$ satisfies the Poincar\'e duality part of Theorem \ref{TheoremChowKahlerPackage},
then $\underline{\theta}_i$ is injective.
\end{proposition}

\begin{proof}
The proof is essentially identical to that of Proposition \ref{PropositionPushforwardI}.
\end{proof}

For a flat $F$ in $\mathscr{S}_i$, 
we write $\theta_i^{F \cup i}$  for the pullback map
between the augmented Chow rings obtained from the deletion of $i$ from the localization $\mathrm{M}^{F \cup i}$:
\[
\theta_i^{F \cup i}: \mathrm{CH}(\mathrm{M}^F)\to \mathrm{CH}(\mathrm{M}^{F\cup i}).  
\]
Similarly, for a flat $F$ in $\underline{\mathscr{S}}_i$, 
we write $\underline{\theta}^{F\cup i}_i$ for the pullback map between the Chow rings obtained from the deletion of $i$ from the localization $\mathrm{M}^{F \cup i}$:
\[
\underline{\theta}_i^{F \cup i}: \underline{\mathrm{CH}}(\mathrm{M}^F)\to \underline{\mathrm{CH}}(\mathrm{M}^{F\cup i}).
\]
Note that $i$ is  a coloop of $\mathrm{M}^{F \cup i}$ in these cases.

\begin{proposition}\label{lem:top degree vanishing}
The summands appearing in Theorems  \ref{TheoremUnderlinedDecomposition} and \ref{TheoremDecomposition} can be described as follows.
\begin{enumerate}[(1)]\itemsep 5pt
\item If $F \in \mathscr{S}_i$, 
then $x_{F\cup i}\mathrm{CH}_{(i)} = \psi_\mathrm{M}^{F\cup i}\big(\underline{\mathrm{CH}}(\mathrm{M}_{F\cup i}) \otimes \theta_i^{F \cup i}\mathrm{CH}(\mathrm{M}^{F})\big)$.
\item If $F \in \underline{\mathscr{S}}_i$, 
then $x_{F\cup i}\underline{\mathrm{CH}}_{(i)} = \underline{\psi}_\mathrm{M}^{F\cup i}\big(\underline{\mathrm{CH}}(\mathrm{M}_{F\cup i}) \otimes  \underline{\theta}_i^{F \cup i}\underline{\mathrm{CH}}(\mathrm{M}^{F})\big)$.
\item If $i$ is a coloop of $\mathrm{M}$, then 
$x_{E \setminus i} \CH_{(i)}=\psi^{E \setminus i}_\mathrm{M} \CH(\mathrm{M} \setminus i)$
and
$x_{E \setminus i} \underline{\CH}_{(i)}=\underline{\psi}^{E \setminus i}_\mathrm{M} \underline{\CH}(\mathrm{M} \setminus i)$.
\end{enumerate}
\end{proposition}

It follows, assuming Poincar\'e duality for the Chow rings,\footnote{We need Poincar\'e duality for $\uCH(\mathrm{M}^F)$, $\CH(\mathrm{M}^F)$,  $\uCH(\mathrm{M}^{F \cup i})$, $\CH(\mathrm{M}^{F \cup i})$,  and
$\uCH(\mathrm{M}_{F \cup i})$.} that
\[
x_{F\cup i}\mathrm{CH}_{(i)}  \cong \underline{\mathrm{CH}}(\mathrm{M}_{F\cup i}) \otimes \mathrm{CH}(\mathrm{M}^{F})[-1]
\ \ \text{and} \ \ 
x_{F\cup i}\underline{\mathrm{CH}}_{(i)}  \cong \underline{\mathrm{CH}}(\mathrm{M}_{F\cup i}) \otimes \underline{\mathrm{CH}}(\mathrm{M}^{F})[-1].
\]
Therefore, again assuming Poincar\'e duality for the Chow rings,  we have
\[
\dim  x_{F\cup i}\underline{\mathrm{CH}}^{k-1}_{(i)}  = \dim  x_{F\cup i}\underline{\mathrm{CH}}_{(i)}^{d-k-2}
\ \ \text{and} \ \ 
\dim  x_{F\cup i}\mathrm{CH}^{k-1}_{(i)}  = \dim  x_{F\cup i}\mathrm{CH}_{(i)}^{d-k-1}.
\]


\begin{proof} 
We  prove the first statement. The proof of the second statement is essentially identical. 
The third statement is a straightforward consequence of  the fact that  $\varphi^{E \setminus i}_\mathrm{M} \circ \theta_i$ 
and $\underline{\varphi}^{E \setminus i}_\mathrm{M} \circ \underline{\theta}_i$ are the identity maps when $i$ is a coloop.

Let $F$ be a flat in $\mathscr{S}_i$.
It is enough to show that 
\[
\varphi^{F\cup i}_{\mathrm{M}} \big(\mathrm{CH}_{(i)}\big) = \underline{\mathrm{CH}}(\mathrm{M}_{F\cup i}) \otimes \theta_i^{F\cup i} \mathrm{CH}(\mathrm{M}^{F}),
\]
since the result will then follow by applying $\psi_{\mathrm{M}}^{F\cup i}$.
The projection $\pi_i$ maps the ray  $\rho_{F\cup i}$ to the ray $\rho_{F}$,
and hence $\pi_i$ defines  morphisms of fans
\[
\begin{tikzcd}[column sep=10mm, row sep=10mm]
\text{star}_{\rho_{F\cup i}} \Pi_\mathrm{M} \arrow[d, "\pi'_i"'] &\arrow[l, "\iota_{F \cup i}"']  \underline{\Pi}_{\mathrm{M}_{F\cup i}}\times \Pi_{\mathrm{M}^{F\cup i}}  \arrow[d, "\pi''_i"'] & \arrow[l, equal]\underline{\Pi}_{(\mathrm{M}/i)_{F}}\times \Pi_{\mathrm{M}^{F\cup i}} \arrow[d, "\pi'''_i"']\\
\text{star}_{\rho_{F}} \Pi_{\mathrm{M}\setminus i} &  \arrow[l, "\iota_F"'] \underline{\Pi}_{(\mathrm{M}\setminus i)_{F}}\times \Pi_{(\mathrm{M} \setminus i)^F} & \arrow[l, equal] \underline{\Pi}_{(\mathrm{M}\setminus i)_{F}}\times \Pi_{\mathrm{M}^F},
\end{tikzcd}
\]
where $\iota_{F \cup i}$ and $\iota_F$ are the isomorphisms in Proposition \ref{PropositionStars}.
The main point is that the matroid $(\mathrm{M}/i)_F$ is a quotient of $(\mathrm{M}\setminus i)_F$.
In other words, we have the inclusion of Bergman fans
\[
\underline{\Pi}_{(\mathrm{M}/i)_F} \subseteq \underline{\Pi}_{(\mathrm{M}\setminus i)_F}.
\]
Therefore, the morphism $\pi_i'''$ admits the factorization
\[
\begin{tikzcd}[column sep=10mm, row sep=10mm]
\underline{\Pi}_{(\mathrm{M}/i)_{F}}\times \Pi_{\mathrm{M}^{F\cup i}} \arrow[r,two heads] &
\underline{\Pi}_{(\mathrm{M}/i)_{F}}\times \Pi_{\mathrm{M}^{F}}  \arrow[r,hook]&
\underline{\Pi}_{(\mathrm{M}\setminus i)_{F}}\times \Pi_{\mathrm{M}^F},
\end{tikzcd}
\]
where the second map induces a surjective pullback map $q$ between the Chow rings. 
By the equality $(\mathrm{M}/i)_F = \mathrm{M}_{F \cup i}$, we have the commutative diagram of pullback maps between the Chow rings
\[
\begin{tikzcd}[column sep=10mm, row sep=10mm]
\mathrm{CH}(\mathrm{M}\setminus i) \arrow[rr, "\theta_i"] \arrow[d,two heads, "\varphi^F_{\mathrm{M} \setminus i}"']&& \mathrm{CH}(\mathrm{M})  \arrow[d,two heads, "\varphi^{F \cup i}_\mathrm{M}"]\\
\underline{\mathrm{CH}}((\mathrm{M}\setminus i)_F)\otimes \mathrm{CH}((\mathrm{M} \setminus i)^F) \arrow[r,two heads, "q"]&
\underline{\mathrm{CH}}(\mathrm{M}_{F\cup i})\otimes \mathrm{CH}(\mathrm{M}^F) \arrow[r,"1 \otimes \theta^{F \cup i}_i"]
& \underline{\mathrm{CH}}(\mathrm{M}_{F\cup i})\otimes \mathrm{CH}(\mathrm{M}^{F\cup i}).
\end{tikzcd}
\]
The conclusion follows from the surjectivity of the pullback maps $\varphi^F_{\mathrm{M} \setminus i}$ and $q$.
\end{proof}

\begin{remark}\label{rem:top degree vanishing}
Since  $i$ is a coloop in $\mathrm{M}^{F \cup i}$ when $F\in \mathscr{S}_i$ or $F\in \underline{\mathscr{S}}_i$, 
 Proposition \ref{lem:top degree vanishing} implies that
\[
x_{F\cup i}\mathrm{CH}_{(i)}^{d-1}=0\; \text{for $F\in \mathscr{S}_i$} 
\ \ \text{and} \ \ 
 x_{F\cup i}\underline{\mathrm{CH}}_{(i)}^{d-2}=0\; \text{for $F\in \underline{\mathscr{S}}_i$}.
\]
\end{remark}


\begin{proposition}\label{lemma_bothdegree}
The Poincar\'e pairing on the summands  appearing in Theorems  \ref{TheoremUnderlinedDecomposition} and \ref{TheoremDecomposition} can be described as follows.
\begin{enumerate}[(1)]\itemsep 5pt
\item If $F\in \mathscr{S}_i$, 
then for any $\mu_1, \mu_2\in \underline{\mathrm{CH}}(\mathrm{M}_{F\cup i}) \otimes {\mathrm{CH}}(\mathrm{M}^{F})$ of complementary degrees,
\[
\deg_{\mathrm{M}}\big(\psi^{F\cup i}_{\mathrm{M}}\big(1 \otimes \theta_i^{F \cup i}(\mu_1)\big)\cdot \psi^{F\cup i}_{\mathrm{M}}\big(1 \otimes \theta_i^{F \cup i}(\mu_2)\big) \big)
=-\underline\deg_{\mathrm{M}_{F\cup i}} \otimes \deg_{\mathrm{M}^{F}}(\mu_1 \mu_2).
\]
\item If $F\in \underline{\mathscr{S}}_i$, 
then for any $\nu_1, \nu_2\in \underline{\mathrm{CH}}(\mathrm{M}_{F\cup i}) \otimes \underline{\mathrm{CH}}(\mathrm{M}^{F})$ of complementary degrees,
\[
\underline\deg_{\mathrm{M}}\big(\underline\psi^{F\cup i}_{\mathrm{M}}\big( 1 \otimes  \underline\theta_i^{F \cup i}(\nu_1)\big)\cdot \underline\psi^{F\cup i}_{\mathrm{M}}\big( 1 \otimes  \underline\theta_i^{F \cup i}(\nu_2)\big) \big)
=-\underline\deg_{\mathrm{M}_{F\cup i}} \otimes \underline\deg_{\mathrm{M}^{F}}( \nu_1\nu_2).
\]
\end{enumerate}
\end{proposition}

It follows, assuming Poincar\'e duality for the Chow rings,\footnote{We need Poincar\'e duality for $\uCH(\mathrm{M}^F)$, $\CH(\mathrm{M}^F)$,  $\uCH(\mathrm{M}^{F \cup i})$, $\CH(\mathrm{M}^{F \cup i})$,  and
$\uCH(\mathrm{M}_{F \cup i})$.} that
the restriction of the Poincar\'e pairing of $\CH(\mathrm{M})$ to the subspace $x_{F \cup i} \CH_{(i)}$ is non-degenerate, 
and the restriction of the Poincar\'e pairing of $\uCH(\mathrm{M})$ to the subspace  $x_{F \cup i} \uCH_{(i)}$
is non-degenerate. 

\begin{proof}
We prove the first identity.
The second identity can be proved in the same way.

Since the pushforward $\psi^{F\cup i}_{\mathrm{M}}$ is  a $\mathrm{CH}(\mathrm{M})$-module homomorphism, 
the left-hand side is
\[
\deg_{\mathrm{M}} \big(\psi^{F\cup i}_{\mathrm{M}}\big(\varphi^{F\cup i}_{\mathrm{M}}\psi^{F\cup i}_{\mathrm{M}}\big( 1\otimes \theta_i^{F \cup i}(\mu_1)\big)\cdot \big(1\otimes \theta_i^{F \cup i}(\mu_2)\big) \big)\big).
\]
The pushforward commutes with the degree maps, so the above is equal to
\[
\underline\deg_{\mathrm{M}_{F\cup i}}\otimes \deg_{\mathrm{M}^{F\cup i}} \big(\varphi^{F\cup i}_{\mathrm{M}}\psi^{F\cup i}_{\mathrm{M}}\big( 1\otimes \theta_i^{F \cup i}(\mu_1)\big)\cdot \big( 1 \otimes \theta_i^{F \cup i}(\mu_2)\big)\big).
\]
Using that the composition $\varphi^{F\cup i}_{\mathrm{M}}\psi^{F\cup i}_{\mathrm{M}}$ is multiplication by $\varphi^{F\cup i}_{\mathrm{M}}(x_{F \cup i})$, we get
\[
-\underline\deg_{\mathrm{M}_{F\cup i}}\otimes \deg_{\mathrm{M}^{F\cup i}} \big( \big(1 \otimes \alpha_{\mathrm{M}^{F \cup i}} + \underline{\beta}_{\mathrm{M}_{F \cup i}} \otimes 1\big) \cdot \big( 1\otimes \theta_i^{F \cup i}(\mu_1)\big)\cdot \big( 1 \otimes \theta_i^{F \cup i}(\mu_2)\big)\big).
\]
Since $i$ is a coloop of $\mathrm{M}^{F \cup i}$, the expression simplifies to
\[
-\underline\deg_{\mathrm{M}_{F\cup i}}\otimes \deg_{\mathrm{M}^{F\cup i}} \big( \big(1 \otimes \alpha_{\mathrm{M}^{F \cup i}} \big) \cdot \big(1\otimes \theta_i^{F \cup i}(\mu_1)\big)\cdot \big( 1 \otimes \theta_i^{F \cup i}(\mu_2)\big)\big).
\]
Now the third part of Lemma \ref{lemma_PoincareCompatible} shows that the above quantity is 
the right-hand side of the formula in statement (1).
\end{proof}


\begin{lemma}\label{lem:multiplying summands}
If flats $F_1$, $F_2$ are in $\mathscr{S}_i$ and $F_1$ is a proper subset of $F_2$, then 
\[
x_{F_1\cup i}\,x_{F_2\cup i} \in x_{F_1\cup i}\mathrm{CH}_{(i)}.
\]
Similarly, if $F_1$, $F_2$ are in $\underline{\mathscr{S}}_i$ and $F_1$ is a proper subset of $F_2$, then 
\[
x_{F_1\cup i}\,x_{F_2\cup i} \in x_{F_1\cup i}\underline{\mathrm{CH}}_{(i)}.
\]
\end{lemma}

\begin{proof}
Since $F_1\cup i$ is not comparable to $F_2$, we have
\[x_{F_1\cup i}\,x_{F_2\cup i} = x_{F_1\cup i}(x_{F_2}+x_{F_2\cup i}) = x_{F_1\cup i}\theta_i(x_{F_2}).\]
The second part follows from the same argument. 
\end{proof}




\begin{proof}[Proof of Theorem \ref{TheoremUnderlinedDecomposition}, Theorem \ref{TheoremDecomposition}, and 
parts (1) and (4) of Theorem \ref{TheoremChowKahlerPackage}]
All the summands in the proposed decompositions are cyclic, and therefore indecomposable in the category of graded  modules.\footnote{By \cite[Corollary 2]{CF} or \cite[Theorem 3.2]{GG},
the indecomposability of the summands in the category of graded modules implies the indecomposability of the summands in the category of  modules.}
We prove the decompositions by induction on the cardinality of the ground set $E$.
If $E$ is empty, then Theorem \ref{TheoremUnderlinedDecomposition}, Theorem \ref{TheoremDecomposition}, and part (1) of 
Theorem \ref{TheoremChowKahlerPackage}
are vacuous, while part (4) of Theorem \ref{TheoremChowKahlerPackage} is trivial.  Furthermore, all of these results are trivial when $E$
is a singleton.  Thus, we may assume
that $i$ is an element of $E$, that $E\setminus i$ is nonempty,
and that all the results hold for loopless matroids whose ground set is a proper subset of $E$.

First we assume that $i$ is not a coloop.  Let us show that the terms in the right-hand side of the decomposition \eqref{eqn:deletion decomposition} are orthogonal.  Multiplying $\mathrm{CH}_{(i)}$ and $x_{F\cup i}\mathrm{CH}_{(i)}$ lands in $x_{F\cup i}\mathrm{CH}_{(i)}$, and this ideal
vanishes in degree $d$ by Remark \ref{rem:top degree vanishing}, so they are orthogonal.  On the other hand, the product of $x_{F_1\cup i}\mathrm{CH}_{(i)}$ and 
$x_{F_2\cup i}\mathrm{CH}_{(i)}$ vanishes if $F_1, F_2\in \mathscr{S}_i$ are not comparable, while if $F_1 < F_2$ or $F_2 < F_1$, the product is contained in $x_{F_1\cup i}\mathrm{CH}_{(i)}$ or $x_{F_2\cup i}\mathrm{CH}_{(i)}$ respectively, by Lemma \ref{lem:multiplying summands}. So these terms are also orthogonal.

It follows from the induction hypothesis and Lemma \ref{lemma_PoincareCompatible} that the restriction of the Poincar\'e pairing of $\mathrm{CH}(\mathrm{M})$ to $\mathrm{CH}_{(i)}$ is non-degenerate. By Proposition \ref{lem:top degree vanishing}, Proposition \ref{lemma_bothdegree}, and the induction hypothesis, the restriction of the Poincar\'e pairing of $\mathrm{CH}(\mathrm{M})$ to any other summand $x_{F\cup i}\mathrm{CH}_{(i)}$ is also non-degenerate. Therefore, we can conclude that the sum on the right-hand side of \eqref{eqn:deletion decomposition} is a direct sum with a non-degenerate Poincar\'e pairing.

To complete the proof of the decomposition \eqref{eqn:deletion decomposition} and the Poincar\'e duality theorem for $\mathrm{CH}(\mathrm{M})$, 
we must show that the direct sum
$$\mathrm{CH}_{(i)} \oplus \bigoplus_{F \in \mathscr{S}_i} x_{F\cup i} \mathrm{CH}_{(i)}$$
is equal to all of $\mathrm{CH}(\mathrm{M})$.
This is obvious in degree $0$.  To see that it holds in degree $1$, 
it is enough to check that  $x_G$ is contained in the direct sum for any proper flat $G$ of $\mathrm{M}$.
If $G \setminus i$ is a not flat of $\mathrm{M}$, then
$
x_G=\theta_i(x_{G \setminus i}).
$
If $G \setminus i$ is a flat of $\mathrm{M}$, then either $G \setminus i \in \mathscr{S}_i$
or $G \in \mathscr{S}_i$. In the first case, $x_G$ is an element of the summand indexed by $G \setminus i$.
In the second case,
$
x_G = \theta_i(x_G) - x_{G\cup i} \in \mathrm{CH}_{(i)} + x_{G\cup i}\mathrm{CH}_{(i)}.
$

Since our direct sum is a sum of $\mathrm{CH}(\mathrm{M}\setminus i)$-modules and it includes the degree $0$ and $1$ parts of $\mathrm{CH}(\mathrm{M})$, 
it will suffice to show that $\mathrm{CH}(\mathrm{M})$ is generated in degrees $0$ and $1$ as a graded $\mathrm{CH}(\mathrm{M}\setminus i)$-module. In other words, we need to show that
\[
\mathrm{CH}^1_{(i)}\cdot \mathrm{CH}^k(\mathrm{M})=\mathrm{CH}^{k+1}(\mathrm{M}) \ \ \text{for any $k \ge 1$.}
\]

We first prove the equality when $k=1$. 
Since we have proved that the decomposition \eqref{eqn:deletion decomposition} holds in degree $1$, we know that
\[
\mathrm{CH}^2(\mathrm{M}) = \mathrm{CH}^1(\mathrm{M})\cdot \mathrm{CH}^1(\mathrm{M})
= \left(\mathrm{CH}^1_{(i)} \oplus\bigoplus_{F \in \mathscr{S}_i} \mathbb{Q} x_{F\cup i}\right)
\cdot \left(\mathrm{CH}^1_{(i)} \oplus\bigoplus_{F \in \mathscr{S}_i} \mathbb{Q} x_{F\cup i}\right).
\]
Using Lemma \ref{lem:multiplying summands}, we may reduce the problem to showing that 
\[
x_{F\cup i}^2\in \mathrm{CH}^1_{(i)}\cdot \mathrm{CH}^1(\mathrm{M}) \ \ \text{for any $F\in \mathscr{S}_i$.}
\]
We can rewrite the relation $0 = x_{F}y_i$ in the augmented Chow ring of $\mathrm{M}$ as
\begin{align*}
0 & = (\theta_i(x_{F}) - x_{F\cup i})\sum_{i \notin G} x_G\\
& = \theta_i(x_{F})\Big(\sum_{i \notin G} x_G\Big) - x_{F\cup i}\Big(\sum_{G\leq F} x_G\Big),\\
& = \theta_i(x_{F})\Big(\sum_{i \notin G} x_G\Big) - (\theta_i(x_{F})-x_{F})\Big(\sum_{G< F} x_G\Big)-x_{F\cup i} x_{F}\\
& = \theta_i(x_{F})\Big(\sum_{i \notin G} x_G- \sum_{G< F} x_G\Big)+ x_{F}\Big(\sum_{G<F} x_G\Big)-x_{F\cup i}\theta_i(x_{F})+x_{F\cup i}^2, 
\end{align*}
thus reducing the problem to showing that 
\[
x_F x_G\in \mathrm{CH}^1_{(i)}\cdot \mathrm{CH}^1(\mathrm{M}) \ \ \text{for any $G < F\in \mathscr{S}_i$.}
\]
The collection $\mathscr{S}_i$ is downward closed, meaning that if $G<F\in \mathscr{S}_i$, then $G\in \mathscr{S}_i$;
therefore, 
\[
x_F x_G = (\theta_i(x_F) - x_{F\cup i})(\theta_i(x_G) - x_{G\cup i}).
\]
Lemma \ref{lem:multiplying summands} tells us that  $x_{F\cup i}x_{G\cup i}\in \mathrm{CH}^1_{(i)}\cdot \mathrm{CH}^1(\mathrm{M})$,
thus so is $x_Fx_G$.

We next prove the equality when  $k\geq 2$.
In this case, we use the result for $k=1$ along with the fact that the algebra $\mathrm{CH}(\mathrm{M})$ is generated in degree $1$ to conclude that
\[
\mathrm{CH}^1_{(i)}\cdot \mathrm{CH}^{k}(\mathrm{M})=\mathrm{CH}^1_{(i)} \cdot \mathrm{CH}^{1}(\mathrm{M}) \cdot \mathrm{CH}^{k-1}(\mathrm{M})=\mathrm{CH}^{2}(\mathrm{M}) \cdot \mathrm{CH}^{k-1}(\mathrm{M})=\mathrm{CH}^{k+1}(\mathrm{M}).
\]
This completes the proof of 
the decomposition \eqref{eqn:deletion decomposition} and the Poincar\'e duality theorem for $\mathrm{CH}(\mathrm{M})$ when there is an element $i$ that is not a coloop of $\mathrm{M}$.

The proof when $i$ is a coloop is almost the same; we explain the places where something different must be said.  The orthogonality of $x_{E \setminus i}\mathrm{CH}_{(i)}$ and $x_{F\cup i} \mathrm{CH}_{(i)}$ for $F \in \mathscr{S}_i$ follows because $E \setminus i$ and $F\cup i$ are incomparable.  To show that the right-hand side of \eqref{eqn:deletion decomp coloop} spans $\mathrm{CH}(\mathrm{M})$, one extra statement we need to check is that
\[
x_{E\setminus i}^2\in \mathrm{CH}^1_{(i)}\cdot \mathrm{CH}^1(\mathrm{M}).
\] 
Since $i$ is a coloop, $\mathscr{S}_i$ is the set of all flats properly contained in $E\setminus i$, and we have
\[0 = x_{E\setminus i}y_i = \sum_{i\notin F} x_F x_{E\setminus i} = x_{E\setminus i}^2 + \sum_{F\in \mathscr{S}_i}
 x_{E\setminus i}x_F = x_{E\setminus i}^2 + 
\sum_{F\in \mathscr{S}_i} x_{E\setminus i}\theta_i(x_F),\]
where the last equality follows because $E \setminus i$ and $F \cup i$ are not comparable.  Thus
\[
x_{E\setminus i}^2= - \sum_{F\in \mathscr{S}_i} x_{E\setminus i}\theta_i(x_F) \in \mathrm{CH}^1_{(i)}\cdot \mathrm{CH}^1(\mathrm{M}).
\]

By the induction hypothesis, we know $\mathrm{CH}(\mathrm{M}\setminus i)$ satisfies the Poincar\'e duality theorem. By the coloop case of Lemma \ref{lemma_PoincareCompatible}, the Poincar\'e pairing on $\mathrm{CH}(\mathrm{M})$ restricts to a perfect pairing between $\mathrm{CH}_{(i)}$ and $x_{E\setminus i}\mathrm{CH}_{(i)}$. Since $\mathrm{CH}_{(i)}$ is a subring of $\mathrm{CH}(\mathrm{M})$ and is zero in degree $d$, the restriction of the Poincar\'e pairing on $\mathrm{CH}(\mathrm{M})$ to $\mathrm{CH}_{(i)}$ is zero. Therefore, the subspaces $\mathrm{CH}_{(i)}$ and $x_{E\setminus i}\mathrm{CH}_{(i)}$ intersect trivially, and the restriction of the Poincar\'e pairing on $\mathrm{CH}(\mathrm{M})$ to $\mathrm{CH}_{(i)}\oplus x_{E\setminus i}\mathrm{CH}_{(i)}$ is non-degenerate. This completes the proof of the theorems about $\mathrm{CH}(\mathrm{M})$ when $i$ is a coloop. 

We observe that the surjectivity of 
the pullback $\varphi^\varnothing_{\mathrm{M}}$ gives the equality
\[
\underline{\mathrm{CH}}^1_{(i)}\cdot \underline{\mathrm{CH}}^{k}(\mathrm{M})=\underline{\mathrm{CH}}^{k+1}(\mathrm{M}) \ \ \text{for any $k \ge 1$.}
\]
The proof of the theorems about $\underline{\mathrm{CH}}(\mathrm{M})$ then follows by an argument identical to the one used for $\mathrm{CH}(\mathrm{M})$.
\end{proof}

\section{Proofs of the hard Lefschetz theorems and the Hodge--Riemann relations}\label{Section4}

In this section, we prove Theorem \ref{TheoremChowKahlerPackage}.
Parts (1) and (4) have already been proved in the previous section.  We will first prove parts (2) and (3) by induction on the cardinality of $E$.
The proof of parts (5) and (6) is nearly identical to the proof of parts (2) and (3), with the added nuance that 
we use parts (2) and (3) for the matroid $\mathrm{M}$ in the proof of parts (5) and (6) for the matroid $\mathrm{M}$.

For any fan $\Sigma$, we will say that $\Sigma$ satisfies the hard Lefschetz theorem or the Hodge--Riemann
relations with respect to some piecewise linear function on $\Sigma$ if the ring $\mathrm{CH}(\Sigma)$ satisfies
the hard Lefschetz theorem or the Hodge--Riemann relations with respect to the corresponding element of $\mathrm{CH}^1(\Sigma)$.

\begin{proof}[Proof of Theorem \ref{TheoremChowKahlerPackage}, parts (2) and (3)]
The statements are trivial when the cardinality of $E$ is $0$ or $1$, so we will assume throughout the proof that the cardinality
of $E$ is at least $2$.

Let $\mathrm{B}$ be the Boolean matroid on $E$.  By the induction hypothesis, we know that 
for every nonempty proper flat $F$ of $\mathrm{M}$, the fans $\underline{\Pi}_{\mathrm{M}_F}$
and $\underline{\Pi}_{\mathrm{M}^F}$ satisfy the hard Lefschetz theorem and the Hodge--Riemann relations
with respect to any strictly convex piecewise linear functions on $\underline{\Pi}_{\mathrm{B}_F}$ and $\underline{\Pi}_{\mathrm{B}^F}$,
respectively.  By \cite[Proposition 7.7]{AHK}, this implies that for every nonempty proper flat $F$ of $\mathrm{M}$,
the product $\underline{\Pi}_{\mathrm{M}_F}\times\underline{\Pi}_{\mathrm{M}^F}$
satisfies the hard Lefschetz theorem and the Hodge--Riemann relations
with respect to any strictly convex piecewise linear function on $\underline{\Pi}_{\mathrm{B}_F} \times \underline{\Pi}_{\mathrm{B}^F}$.
In other words, $\underline{\Pi}_\mathrm{M}$ satisfies the \emph{local Hodge--Riemann relations} \cite[Definition 7.14]{AHK}:
\[
\text{The star of any ray in $\underline{\Pi}_\mathrm{M}$ satisfies the Hodge--Riemann relations.}
\]
This in turn implies that $\underline{\Pi}_\mathrm{M}$ satisfies the hard Lefschetz theorem 
with respect to any strictly convex piecewise linear function on 
$\underline{\Pi}_\mathrm{B}$ \cite[Proposition 7.15]{AHK}.  It remains to prove only that 
$\underline{\Pi}_\mathrm{M}$ satisfies the Hodge--Riemann relations with respect to any strictly convex piecewise linear function on $\underline{\Pi}_\mathrm{B}$.

Let $\ell$ be a piecewise linear function on $\underline{\Pi}_\mathrm{B}$, and let $\underline{\mathrm{HR}}_\ell^k(\mathrm{M})$ be the \emph{Hodge--Riemann form} 
\[
\underline{\mathrm{HR}}_\ell^k(\mathrm{M}): \underline{\mathrm{CH}}^k(\mathrm{M}) \times \underline{\mathrm{CH}}^k(\mathrm{M}) \longrightarrow \mathbb{Q}, \qquad (\eta_1,\eta_2) \longmapsto (-1)^k\underline{\deg}_\mathrm{M}(\ell^{d-2k-1}\eta_1\eta_2).
\]
By \cite[Proposition 7.6]{AHK}, the fan $\underline{\Pi}_\mathrm{M}$ satisfies the Hodge--Riemann relations with respect to $\ell$ 
if and only if, for all $k < \frac{d}{2}$,
the Hodge--Riemann form $\underline{\mathrm{HR}}_\ell^k(\mathrm{M})$ is non-degenerate and has the signature
\[
\sum_{j=0}^k (-1)^{k-j}\Big( \dim \underline{\mathrm{CH}}^j(\mathrm{M}) - \dim \underline{\mathrm{CH}}^{j-1}(\mathrm{M}) \Big).
\]
Since  $\underline{\Pi}_\mathrm{M}$ satisfies the hard Lefschetz theorem with respect to any strictly convex piecewise linear function on 
$\underline{\Pi}_\mathrm{B}$ and signature is a locally constant function on the space of nonsingular forms,
the following statements are equivalent:
\begin{enumerate}[(1)]\itemsep 5pt
\item[(i)] The fan $\underline{\Pi}_\mathrm{M}$ satisfies the Hodge--Riemann relations with respect to any strictly convex piecewise linear function on $\underline{\Pi}_\mathrm{B}$.
\item[(ii)] The fan $\underline{\Pi}_\mathrm{M}$ satisfies the Hodge--Riemann relations with respect to some strictly convex piecewise linear function on $\underline{\Pi}_\mathrm{B}$.
\end{enumerate}
Furthermore, since satisfying the Hodge--Riemann relations with respect to a given piecewise
linear function is an open condition on the function, statement (ii) is equivalent to the following:
\begin{enumerate}[(1)]\itemsep 5pt
\item[(iii)] The fan $\underline{\Pi}_\mathrm{M}$ satisfies the Hodge--Riemann relations with respect to some convex piecewise linear function on $\underline{\Pi}_\mathrm{B}$.
\end{enumerate}
We show that statement (iii) holds using the semi-small decomposition in Theorem \ref{TheoremUnderlinedDecomposition}.

If $\mathrm{M}$ is the Boolean matroid $\mathrm{B}$,
then $\underline{\mathrm{CH}}(\mathrm{M})$ can be identified with the cohomology ring of the 
smooth complex projective toric variety $X_{\underline{\Pi}_\mathrm{B}}$.
Therefore, in this case, Theorem \ref{TheoremChowKahlerPackage} is a special case of the usual hard Lefschetz theorem and the Hodge--Riemann relations for smooth complex projective varieties.\footnote{It is not difficult to directly prove the hard Lefschetz theorem and the Hodge--Riemann relations for $\underline{\mathrm{CH}}(\mathrm{B})$ using the coloop case of Theorem \ref{TheoremUnderlinedDecomposition}. Alternatively, we may apply McMullen's hard Lefschetz theorem and Hodge--Riemann relations for polytope algebras \cite {McMullen} to the standard permutohedron in $\mathbb{R}^E$.
}

If $\mathrm{M}$ is not the Boolean matroid $\mathrm{B}$, choose an element $i$ that is not a coloop in $\mathrm{M}$,
and consider the morphism of fans
\[
\underline{\pi}_i:\underline{\Pi}_{\mathrm{M}} \longrightarrow \underline{\Pi}_{\mathrm{M} \setminus i}.
\]
By induction, we know that $\underline{\Pi}_{\mathrm{M} \setminus i}$ satisfies the Hodge--Riemann relations with respect to any strictly convex piecewise linear function $\ell$ on $\underline{\Pi}_{\mathrm{B} \setminus i}$.
We will show that $\underline{\Pi}_\mathrm{M}$ satisfies the Hodge--Riemann relations with respect to the pullback  $\ell_i\coloneq \ell \circ \pi_i$, which is a piecewise linear function on $\underline{\Pi}_\mathrm{B}$ that is convex but not necessarily strictly convex.

By Theorem \ref{TheoremUnderlinedDecomposition}, we have the orthogonal decomposition of 
$\underline{\mathrm{CH}}(\mathrm{M})$ into $\underline{\mathrm{CH}}(\mathrm{M}\setminus i)$-modules
\[
\underline{\mathrm{CH}}(\mathrm{M}) = \underline{\mathrm{CH}}_{(i)} \oplus \bigoplus_{F \in \underline{\mathscr{S}}_i} x_{F\cup i} \underline{\mathrm{CH}}_{(i)}. 
\]
By orthogonality, it is enough to show that each summand of $\underline{\mathrm{CH}}(\mathrm{M}) $ 
satisfies the Hodge--Riemann relations with respect to $\ell_i$:
\begin{enumerate}[(1)]\itemsep 5pt
\item[(iv)] For every nonnegative integer $k < \frac{d}{2}$, the bilinear form
\[
\underline{\mathrm{CH}}^k_{(i)}  \times \underline{\mathrm{CH}}^k_{(i)} \longrightarrow \mathbb{Q}, \qquad (\eta_1,\eta_2)\longmapsto (-1)^k \underline{\deg}_\mathrm{M}(\ell_i^{d-2k-1} \eta_1\eta_2)
\]
is positive definite on the kernel of multiplication by $\ell_i^{d-2k}$.
\item[(v)]  For every nonnegative integer $k < \frac{d}{2}$, the bilinear form
\[
x_{F\cup i}\underline{\mathrm{CH}}^{k-1}_{(i)}  \times x_{F\cup i}\underline{\mathrm{CH}}^{k-1}_{(i)} \longrightarrow \mathbb{Q}, \qquad (\eta_1,\eta_2)\longmapsto (-1)^k \underline{\deg}_\mathrm{M}(\ell_i^{d-2k-1} \eta_1\eta_2)
\]
is positive definite on the kernel of multiplication by $\ell_i^{d-2k}$.
\end{enumerate}
By Proposition \ref{DeletionInjection}, the homomorphism $\underline{\theta}_i$ restricts to an isomorphism of 
$\underline{\mathrm{CH}}(\mathrm{M}\setminus i)$-modules 
\[
\underline{\mathrm{CH}}(\mathrm{M} \setminus i) \cong \underline{\mathrm{CH}}_{(i)}.
\]
Thus, statement (iv) follows from Lemma \ref{lemma_PoincareCompatible} and the induction hypothesis applied to $\mathrm{M} \setminus i$.
By Propositions \ref{upsi injective}, \ref{DeletionInjection}, and \ref{lem:top degree vanishing}, 
the homomorphisms $\underline{\theta}_i^{F\cup i}$ and $\underline{\psi}^{F\cup i}_{\mathrm{M}}$ 
give a $\underline{\mathrm{CH}}(\mathrm{M}\setminus i)$-module isomorphism
\[
\underline{\mathrm{CH}}(\mathrm{M}_{F \cup i})\otimes  \underline{\mathrm{CH}}(\mathrm{M}^F)  \cong 
\underline{\mathrm{CH}}(\mathrm{M}_{F \cup i})\otimes \underline{\theta}_i^{F\cup i} \underline{\mathrm{CH}}(\mathrm{M}^F)  \cong x_{F \cup i} \underline{\mathrm{CH}}_{(i)}[1].
\]
Note that the pullback of a strictly convex piecewise linear function on $\underline{\Pi}_{\mathrm{B} \setminus i}$ to the star 
\[
\underline{\Pi}_{(\mathrm{B}\setminus i)_{F}} \times \underline{\Pi}_{(\mathrm{B} \setminus i)^F} = \underline{\Pi}_{\mathrm{B}_{F \cup i}} \times \underline{\Pi}_{\mathrm{B}^F} 
\]
 is the class of a strictly convex piecewise linear function.
Therefore,
statement (v) follows from Proposition \ref{lemma_bothdegree} and the induction applied to $\mathrm{M}_{F \cup i}$ and $\mathrm{M}^F$.
\end{proof}

\begin{proof}[Proof of Theorem \ref{TheoremChowKahlerPackage}, parts (5) and (6)]
This proof is nearly identical to the proof of parts (2) and (3).  In that argument, we used the fact that rays of $\underline{\Pi}_{\mathrm{M}}$
are indexed by nonempty proper flats of $\mathrm{M}$ and the star of the ray $\underline{\rho}_F$ is isomorphic
to $\underline{\Pi}_{\mathrm{M}_F}\times\underline{\Pi}_{\mathrm{M}^F}$, which we can show satisfies the hard Lefschetz
theorem and the Hodge--Riemann relations using the induction hypothesis.
When dealing instead with the augmented Bergman fan $\Pi_{\mathrm{M}}$, we have rays indexed by elements of $E$ and rays
indexed by proper flats of $\mathrm{M}$, with 
$$\text{star}_{\rho_i} \Pi_\mathrm{M} \cong \Pi_{\mathrm{M}_{\text{cl}(i)}} \and
\text{star}_{\rho_F} \Pi_\mathrm{M}\cong\underline{\Pi}_{\mathrm{M}_{F}}\times \Pi_{\mathrm{M}^F}.$$
Thus the stars of $\rho_i$ and $\rho_F$ for nonempty $F$ can be shown to satisfy the hard Lefschetz
theorem and the Hodge--Riemann relations using the induction hypothesis.  However, the star of $\rho_\varnothing$
is isomorphic to $\underline{\Pi}_{\mathrm{M}}$, so we need to use parts (2) and (3) of Theorem \ref{TheoremChowKahlerPackage}
for $\mathrm{M}$ itself.
\end{proof}

\begin{remark}\label{RemarkAlternative}
It is possible to deduce Poincar\'e duality, the hard Lefschetz theorem, and the Hodge--Riemann relations for $\mathrm{CH}(\mathrm{M})$  using  \cite[Theorem 6.19 and Theorem 8.8]{AHK},
where the three properties are proved for generalized Bergman fans 
$\Sigma_{\mathrm{N},\mathscr{P}}$  in \cite[Definition 3.2]{AHK}. 
We sketch the argument here, leaving details to the interested readers.
Consider  the direct sum $\mathrm{M} \oplus 0$ of $\mathrm{M}$ and the rank $1$ matroid on  the singleton $\{0\}$ and the order filter  $\mathscr{P}(\mathrm{M})$ of all proper flats of $\mathrm{M} \oplus 0$ that contain $0$.
The symbols  $\mathrm{B}\oplus 0$ and $\mathscr{P}(\mathrm{B})$ are defined in the same way for the Boolean matroid $\mathrm{B}$ on $E$.
It is straightforward to check that the linear isomorphism
\[
\mathbb{R}^E \longrightarrow \mathbb{R}^{E \cup 0}/\langle \mathbf{e}_{E}+\mathbf{e}_0\rangle, \quad \mathbf{e}_j \longmapsto \mathbf{e}_j 
\]
identifies the complete fan $\Pi_\mathrm{B}$ with the complete fan $\Sigma_{\mathrm{B} \hspace{0.3mm}\oplus  \hspace{0.3mm}0, \mathscr{P}(\mathrm{B})}$,
and the augmented Bergman fan $\Pi_\mathrm{M}$ with a subfan of $\Sigma_{\mathrm{M} \hspace{0.3mm}\oplus  \hspace{0.3mm} 0,\mathscr{P}(\mathrm{M})}$. 
The third identity in Lemma \ref{PropositionIdentities} shows that
the inclusion of the augmented Bergman fan $\Pi_\mathrm{M}$ into the generalized Bergman fan $\Sigma_{\mathrm{M} \hspace{0.3mm}\oplus  \hspace{0.3mm} 0,\mathscr{P}(\mathrm{M})}$ induces an isomorphism between their Chow rings.
 \end{remark}


\section{Proof of Theorem \ref{TheoremSimplexDecomposition}}\label{Section5}

In this section, we prove the decomposition (\ref{underlinedalphadecomposition}) by induction on the cardinality of $E$.
The decomposition (\ref{alphadecomposition}) can be proved using the same argument. 
The results are trivial when $E$ has at most one element.
Thus, we may assume that $i$ is an element of $E$, that $E \setminus i$ is nonempty, and that all the results hold for loopless matroids whose ground set is a proper subset of $E$.

We first prove that the summands appearing in the right-hand side of  (\ref{underlinedalphadecomposition})  are orthogonal to each other. 

\begin{lemma}
Let $F$ and $G$ be distinct nonempty proper flats of $\mathrm{M}$.
\begin{enumerate}[(1)]\itemsep 5pt
\item The spaces  $\underline\psi^F_\mathrm{M}\ \uCH(\mathrm{M}_F)\otimes \uJ_{\underline{\alpha}}(\mathrm{M}^F)$ and $\underline{\mathrm{H}}_{\underline{\alpha}}(\mathrm{M})$  are orthogonal in $\underline{\CH}(\mathrm{M})$.
\item The spaces
$\underline{\psi}^F_\mathrm{M} \uCH(\mathrm{M}_F)\otimes \uJ_{\underline{\alpha}}(\mathrm{M}^F)$ and $\underline{\psi}^G_\mathrm{M} \uCH(\mathrm{M}_G)\otimes \uJ_{\underline{\alpha}}(\mathrm{M}^G)$ are orthogonal in $\underline{\CH}(\mathrm{M})$.
\end{enumerate}
\end{lemma}

\begin{proof}
The fifth bullet point in Proposition \ref{DefinitionUnderlinedPull}, together with the fact that $\underline{\psi}_{\mathrm{M}}^F$ is a $\underline{\CH}(\mathrm{M})$-module homomorphism via $\underline{\varphi}_{\mathrm{M}}^F$, 
implies that every summand in the right-hand side of  (\ref{underlinedalphadecomposition}) is an $\underline{\mathrm{H}}_{\underline{\alpha}}(\mathrm{M})$-submodule. 
Thus the first orthogonality follows from the vanishing of  $\underline\psi^F_\mathrm{M} \ \uCH(\mathrm{M}_F)\otimes \uJ_{\underline{\alpha}}(\mathrm{M}^F)$ in degree $d-1$.

For the second orthogonality, we may suppose that $F$ is a proper subset of  $G$.
Since $\underline\psi^G_\mathrm{M}$ is a $\underline{\CH}(\mathrm{M})$-module homomorphism commuting with the degree maps, it is enough to show that 
\[
\text{$\uvarphi^G_\mathrm{M}\underline{\psi}^F_\mathrm{M} \uCH(\mathrm{M}_F)\otimes \uJ_{\underline{\alpha}}(\mathrm{M}^F)$ and $\uCH(\mathrm{M}_G)\otimes \uJ_{\underline{\alpha}}(\mathrm{M}^G)$ are orthogonal in $\uCH(\mathrm{M}_{G})\otimes \uCH(\mathrm{M}^{G})$.}
\]
For this, we use the commutative diagram of pullback and pushforward maps
\[
\xymatrixcolsep{5pc}\xymatrix{
\uCH(\mathrm{M}_{F})\otimes \uCH(\mathrm{M}^{F}) \ar[r]^{\underline\psi^{F}_\mathrm{M}} \ar[d]^{\uvarphi^{G \setminus F}_{\mathrm{M}_F}\otimes\ 1}&\uCH(\mathrm{M})\ar[d]^{\underline\varphi^G_\mathrm{M}} \\
\uCH(\mathrm{M}_{G})\otimes \uCH(\mathrm{M}^{G}_{F})\otimes \uCH(\mathrm{M}^{F}) \ar[r]^{\quad 1\ \otimes \ \underline\psi^F_{\mathrm{M}^G}}&\uCH(\mathrm{M}_{G})\otimes \uCH(\mathrm{M}^{G}),
}
\]
which  further reduces to the assertion that
\[
\text{
$\underline\psi^F_{\mathrm{M}^G}\uCH(\mathrm{M}_F^G)\otimes \uJ_{\underline{\alpha}}(\mathrm{M}^F)$ and  $\uJ_{\underline{\alpha}}(\mathrm{M}^G)$ are 
orthogonal  in $\uCH(\mathrm{M}^G)$.
}
\]
Since $\uJ_{\underline{\alpha}}(\mathrm{M}^G)\subseteq \underline\H_{\underline{\alpha}}(\mathrm{M}^G)$, the above follows from the first orthogonality for $\mathrm{M}^G$.
\end{proof}

We next show that the restriction of the Poincar\'e pairing of $\uCH(\mathrm{M})$ to each summand appearing in the right-hand side of  (\ref{underlinedalphadecomposition}) is non-degenerate. 

\begin{lemma}
Let $F$ be a nonempty proper flat of $\mathrm{M}$, and let $k=\text{rk}_\mathrm{M}(F)$.
\begin{enumerate}[(1)]\itemsep 5pt
\item The restriction of the Poincar\'e pairing of $\uCH(\mathrm{M})$ to $\underline{\mathrm{H}}_{\underline{\alpha}}(\mathrm{M})$ is non-degenerate.
\item The restriction of the Poincar\'e pairing of $\uCH(\mathrm{M})$ to $\underline{\psi}^F_\mathrm{M} \uCH(\mathrm{M}_F)\otimes \uJ_{\underline{\alpha}}(\mathrm{M}^F)$ is non-degenerate.
\end{enumerate}
\end{lemma}

\begin{proof}
The first statement follows from Proposition \ref{PropositionAlphaDegree}.
We prove the second statement.

Since the Poincar\'e pairing of $\underline{\CH}(\mathrm{M}_F)$  is non-degenerate, it is enough to show that the restriction of the Poincar\'e pairing satisfies
\[
\underline\deg_{\mathrm{M}}\big(\upsi^F_\mathrm{M}(\mu_1\otimes \nu_1)\cdot \upsi^F_\mathrm{M}(\mu_2\otimes \nu_2) \big)
=-\underline\deg_{\mathrm{M}_F}(\mu_1 \mu_2)\ \underline\deg_{\mathrm{M}^F}(\underline\alpha_{\mathrm{M}^F}\nu_1\nu_2).
\]
The proof of the identity is nearly identical to that of Proposition \ref{lemma_bothdegree}.
The left-hand side is 
\[
\underline\deg_{\mathrm{M}}\big(\upsi^F_\mathrm{M} \big( \uvarphi^F_\mathrm{M}\upsi^F_\mathrm{M} (\mu_1\otimes \nu_1)\cdot (\mu_2\otimes \nu_2) \big) \big)
=
\underline\deg_{\mathrm{M}_F} \otimes \underline\deg_{\mathrm{M}^F}  \big( \uvarphi^F_\mathrm{M}\upsi^F_\mathrm{M} (\mu_1\otimes \nu_1)\cdot (\mu_2\otimes \nu_2) \big) 
\]
because $\upsi^F_\mathrm{M}$ is a $\uCH(\mathrm{M})$-module homomorphism commuting with the degree maps.
Since the composition $ \uvarphi^F_\mathrm{M}\upsi^F_\mathrm{M} $ is multiplication by $ \uvarphi^F_\mathrm{M}(x_F)$, the above becomes
\[
-\underline\deg_{\mathrm{M}_F} \otimes \underline\deg_{\mathrm{M}^F}  \big( (1 \otimes \underline{\alpha}_{\mathrm{M}^F} +\underline{\beta}_{\mathrm{M}_F} \otimes 1) \cdot (\mu_1\otimes \nu_1)\cdot (\mu_2\otimes \nu_2) \big).
\]
The vanishing of $\uJ_{\underline{\alpha}}(\mathrm{M}^F)$ in degree $k-1$
 further simplifies the above to the desired expression 
\[
-\underline\deg_{\mathrm{M}_F} \otimes \underline\deg_{\mathrm{M}^F}  \big( (1 \otimes \underline{\alpha}_{\mathrm{M}^F} ) \cdot (\mu_1\otimes \nu_1)\cdot (\mu_2\otimes \nu_2) \big) 
=-\underline\deg_{\mathrm{M}_F}(\mu_1 \mu_2)\ \underline\deg_{\mathrm{M}^F}(\underline\alpha_{\mathrm{M}^F}\nu_1\nu_2). \qedhere
\]
\end{proof}


To complete the proof, we only need to show that the graded vector spaces on both sides of  (\ref{underlinedalphadecomposition}) have the same dimension, which is the next proposition. 
\begin{proposition}
As graded vector spaces, there exists an isomorphism
\begin{equation}\label{num_decomposition}
\underline{\CH}(\mathrm{M}) \cong \underline{\mathrm{H}}_{\underline{\alpha}}(\mathrm{M}) \oplus \ \bigoplus_{F \in \underline{\mathscr{C}}(\mathrm{M})}  \  \uCH(\mathrm{M}_F)\otimes \uJ_{\underline{\alpha}}(\mathrm{M}^F)[-1],
\tag{$\underline{\mathrm{D}}_3'$}
\end{equation}
where the sum is over the set $\underline{\mathscr{C}}(\mathrm{M})$ of proper flats of $\mathrm{M}$ with rank at least two.
\end{proposition}
\begin{proof}
We prove the proposition using induction on the cardinality of $E$. Suppose the proposition holds for any matroid whose ground set is a proper subset of $E$. Suppose that there exists an element $i\in E$ that is not a coloop. Then the decomposition (\ref{eqn:underlined deletion decomposition})
implies
\[
\underline{\CH}(\mathrm{M})\cong \uCH(\mathrm{M}\setminus i) \oplus \bigoplus_{G \in \underline{\mathscr{S}}_i(\mathrm{M})} \underline{\mathrm{CH}}(\mathrm{M}_{G\cup i}) \otimes  \underline{\mathrm{CH}}(\mathrm{M}^{G})[-1],
\]
since the maps $\underline{\theta}_i$, $\underline{\theta}_i^{G\cup i}$, and $\underline{\psi}_{\mathrm{M}}^{G\cup i}$ are injective via the Poincar\'{e} duality part of Theorem \ref{TheoremChowKahlerPackage}.
By applying the induction hypothesis to the matroids $\mathrm{M}\setminus i$ and $\mathrm{M}^G$, we see that 
the left-hand side of (\ref{num_decomposition}) is isomorphic to the graded vector space
\begin{align*}
\underline{\mathrm{H}}_{\underline{\alpha}}(\mathrm{M}\setminus i) & \oplus  \bigoplus_{G\in \underline{\mathscr{C}}({\mathrm{M}\setminus i})}   \uCH\big((\mathrm{M}\setminus i)_G\big)\otimes \uJ_{\underline{\alpha}}\big((\mathrm{M}\setminus i)^G\big)[-1]\\
&\oplus   \bigoplus_{G \in \underline{\mathscr{S}}_i({\mathrm{M}})} \ \underline{\mathrm{CH}}(\mathrm{M}_{G\cup i}) \otimes \underline{\mathrm{H}}_{\underline{\alpha}}(\mathrm{M}^G)[-1] \\ &\oplus    \bigoplus_{G \in \underline{\mathscr{S}}_i({\mathrm{M}})}   \bigoplus_{F\in \underline{\mathscr{C}}({\mathrm{M}^G})}  \underline{\mathrm{CH}}(\mathrm{M}_{G\cup i}) \otimes  \uCH(\mathrm{M}^G_F)\otimes \uJ_{\underline{\alpha}}(\mathrm{M}^F)[-2].
\end{align*}
Since $i$ is not a coloop, we may replace
 $ \underline{\H}_{\underline\alpha}(\mathrm{M}\setminus i)$  by
  $\underline{\H}_{\underline\alpha}(\mathrm{M})$.

Now, we further decompose the right-hand side of (\ref{num_decomposition}) to match the displayed expression.
For this, we split the index set $\underline{\mathscr{C}}(\mathrm{M})$ into three groups:
\begin{enumerate}[(1)]\itemsep 5pt
\item $F\in \underline{\mathscr{C}}({\mathrm{M}}), i\in F, F\setminus i \in \underline{\mathscr{S}}_i({\mathrm{M}})$,
\item $F\in \underline{\mathscr{C}}({\mathrm{M}}), i\in F, F\setminus i \notin \underline{\mathscr{S}}_i({\mathrm{M}})$, and
\item $F\in \underline{\mathscr{C}}({\mathrm{M}}), i\notin F$.
\end{enumerate}

Suppose $F$ belongs to the first group.
In this case, we have $\uJ_{\underline\alpha}(\mathrm{M}^{F})\cong\underline{\mathrm{H}}_{\underline{\alpha}}(\mathrm{M}^{F \setminus i})$ as graded vector spaces.
Therefore, we have 
\[
\bigoplus_{\substack{F \in \underline{\mathscr{C}}(\mathrm{M})\\ i\in F, F\setminus i \in \underline{\mathscr{S}}_i(\mathrm{M})}} \underline{\mathrm{CH}}(\mathrm{M}_{F}) \otimes \uJ_{\underline\alpha}(\mathrm{M}^{F})[-1]
\cong 
\bigoplus_{G \in \underline{\mathscr{S}}_i(\mathrm{M})} \underline{\mathrm{CH}}(\mathrm{M}_{G\cup i}) \otimes  \underline{\mathrm{H}}_{\underline{\alpha}}(\mathrm{M}^G)[-1].
\]

Suppose $F$ belongs to the second group.
In this case,  $\mathrm{M}_F=(\mathrm{M}\setminus i)_{F\setminus i}$, 
and the matroids $\mathrm{M}^{F}$ and $(\mathrm{M}\setminus i)^{F\setminus i}$ have the same rank.
Therefore, we have
\[
\bigoplus_{\substack{F \in \underline{\mathscr{C}}(\mathrm{M})\\ i\in F, F\setminus i \notin \underline{\mathscr{S}}_i(\mathrm{M})}} \underline{\mathrm{CH}}(\mathrm{M}_{F}) \otimes \uJ_{\underline\alpha}(\mathrm{M}^{F})[-1]
\cong
\bigoplus_{G\in \underline{\mathscr{C}}({\mathrm{M}\setminus i}) \setminus \underline{\mathscr{C}}(\mathrm{M})}  \  \uCH\big((\mathrm{M}\setminus i)_G\big)\otimes \uJ_{\underline{\alpha}}\big((\mathrm{M}\setminus i)^G\big)[-1].
\]

Suppose $F$ belongs to the third group.  In this case, we apply (\ref{eqn:underlined deletion decomposition}) to 
$\mathrm{M}_F$ and get
\begin{align*}
&\bigoplus_{F \in \underline{\mathscr{C}}(\mathrm{M}), i\notin F}  \uCH(\mathrm{M}_F)\otimes \uJ_{\underline{\alpha}}(\mathrm{M}^F)[-1]  \\
\cong&\bigoplus_{F \in \underline{\mathscr{C}}(\mathrm{M}), i\notin F} \Big(\uCH\big(\mathrm{M}_F\setminus i\big)\oplus \bigoplus_{G\in \underline{\mathscr{S}}_i(\mathrm{M}_F)}\uCH(\mathrm{M}_{G\cup i})\otimes \uCH(\mathrm{M}^G_F)[-1] \Big)\otimes \uJ_{\underline{\alpha}}(\mathrm{M}^F)[-1] \\
\cong&\bigoplus_{F \in \underline{\mathscr{C}}(\mathrm{M}), i\notin F} \uCH\big(\mathrm{M}_F\setminus i\big)\otimes \uJ_{\underline{\alpha}}(\mathrm{M}^F)[-1]\oplus \bigoplus_{\substack{G\in \underline{\mathscr{S}}_i(\mathrm{M})\\ F\in \underline{\mathscr{C}}({\mathrm{M}^G})}}\uCH(\mathrm{M}_{G\cup i})\otimes \uCH(\mathrm{M}^G_F)\otimes \uJ_{\underline{\alpha}}(\mathrm{M}^F)[-2] \\
\cong &\bigoplus_{G\in \underline{\mathscr{C}}({\mathrm{M}\setminus i})\cap \underline{\mathscr{C}}({\mathrm{M}})} \uCH\big((\mathrm{M}\setminus i)_G\big)\otimes \uJ_{\underline{\alpha}}\big((\mathrm{M}\setminus i)^G\big)[-1]\oplus \bigoplus_{\substack{G\in \underline{\mathscr{S}}_i(\mathrm{M})\\ F\in \underline{\mathscr{C}}({\mathrm{M}^G})}}\uCH(\mathrm{M}_{G\cup i})\otimes \uCH(\mathrm{M}^G_F)\otimes \uJ_{\underline{\alpha}}(\mathrm{M}^F)[-2].
\end{align*} 
The decomposition \eqref{num_decomposition} follows. 


Suppose now that every element of $E$ is a coloop of $\mathrm{M}$; that is, $\mathrm{M}$ is a Boolean matroid. We fix an element $i\in E$. The decomposition (\ref{eqn:underlined deletion decomposition coloop}) and the Poincar\'{e} duality part of Theorem \ref{TheoremChowKahlerPackage} imply
\[
\underline{\CH}(\mathrm{M})\cong \uCH(\mathrm{M}\setminus i)\oplus \uCH(\mathrm{M}\setminus i)[-1] \oplus \bigoplus_{G \in \underline{\mathscr{S}}_i(\mathrm{M})} \underline{\mathrm{CH}}(\mathrm{M}_{G\cup i}) \otimes  \underline{\mathrm{CH}}(\mathrm{M}^{G})[-1]. 
\]
The assumption that $i$ is a coloop implies that $\underline{\mathscr{S}}_i(\mathrm{M}) \cap \underline{\mathscr{C}}(\mathrm{M}) =\underline{\mathscr{C}}(\mathrm{M}\setminus i)$.  The induction hypothesis applies to the matroids $\mathrm{M}\setminus i$ and $\mathrm{M}^G$, and hence
the left-hand side of (\ref{num_decomposition}) is isomorphic to 
\begin{align*}
& \underline{\mathrm{H}}_{\underline{\alpha}}(\mathrm{M}\setminus i) \oplus \ \bigoplus_{G\in \underline{\mathscr{C}}(\mathrm{M}\setminus i)}  \  \uCH\big((\mathrm{M}\setminus i)_G\big)\otimes \uJ_{\underline{\alpha}}\big((\mathrm{M}\setminus i)^G\big)[-1]\\
&\oplus \underline{\mathrm{H}}_{\underline{\alpha}}(\mathrm{M}\setminus i)[-1] \oplus \ \bigoplus_{G\in \underline{\mathscr{C}}(\mathrm{M}\setminus i)}  \  \uCH\big((\mathrm{M}\setminus i)_G\big)\otimes \uJ_{\underline{\alpha}}\big((\mathrm{M}\setminus i)^G\big)[-2]\\
&\oplus \bigoplus_{G \in \underline{\mathscr{S}}_i(\mathrm{M})} \underline{\mathrm{CH}}(\mathrm{M}_{G\cup i}) \otimes  \Big(\underline{\mathrm{H}}_{\underline{\alpha}}(\mathrm{M}^G) \oplus \ \bigoplus_{F\in \underline{\mathscr{C}}(\mathrm{M}^G)}  \  \uCH(\mathrm{M}^G_F)\otimes \uJ_{\underline{\alpha}}(\mathrm{M}^F)[-1]\Big)[-1].
\end{align*}

Now, we further decompose the right-hand side of (\ref{num_decomposition}) to match the displayed expression.
For this, we split the index set $\underline{\mathscr{C}}(\mathrm{M})$ into three groups:
\begin{enumerate}[(1)]\itemsep 5pt
\item $F\in \underline{\mathscr{C}}({\mathrm{M}}), i\in F$,
\item $F\in \underline{\mathscr{C}}({\mathrm{M}}), F=E \setminus i$, and
\item $F\in \underline{\mathscr{C}}({\mathrm{M}}), F \in \underline{\mathscr{S}}_i({\mathrm{M}})$.
\end{enumerate}

If $F$ belongs to the first group, then $\uJ_{\underline\alpha}(\mathrm{M}^{F})\cong\underline{\mathrm{H}}_{\underline{\alpha}}(\mathrm{M}^{F \setminus i})$, and hence
\[
\bigoplus_{F\in \underline{\mathscr{C}}(\mathrm{M}), i\in F}  \  \uCH(\mathrm{M}_F)\otimes \uJ_{\underline{\alpha}}(\mathrm{M}^F)[-1]
\cong
\bigoplus_{G \in \underline{\mathscr{S}}_i(\mathrm{M})} \underline{\mathrm{CH}}(\mathrm{M}_{G\cup i}) \otimes  \underline{\mathrm{H}}_{\underline{\alpha}}(\mathrm{M}^G)[-1].
\]
If $F$ is the flat $E \setminus i$,  we have
\[
\underline{\mathrm{H}}_{\underline{\alpha}}(\mathrm{M})\oplus \uCH(\mathrm{M}_{E\setminus i})\otimes \uJ_{\underline{\alpha}}(\mathrm{M}^{E\setminus i})[-1]
\cong
\underline{\mathrm{H}}_{\underline{\alpha}}(\mathrm{M}\setminus i)\oplus \underline{\mathrm{H}}_{\underline{\alpha}}(\mathrm{M}\setminus i)[-1].
\]
If $F$ belongs to the third group, we apply (\ref{eqn:underlined deletion decomposition coloop}) to $\mathrm{M}_F$ and get
\begin{align*}
&\bigoplus_{\substack{F\in \underline{\mathscr{C}}(\mathrm{M})\\ F \in \underline{\mathscr{S}}_i(\mathrm{M})}}  \  \uCH(\mathrm{M}_F)\otimes \uJ_{\underline{\alpha}}(\mathrm{M}^F)[-1]\\
\cong
&\bigoplus_{\substack{F\in \underline{\mathscr{C}}(\mathrm{M})\\F \in \underline{\mathscr{S}}_i(\mathrm{M})}}  \  \Big(\uCH(\mathrm{M}_F\setminus i)\oplus \uCH(\mathrm{M}_F\setminus i)[-1]\oplus \bigoplus_{G\in \underline{\mathscr{S}}_i(\mathrm{M}_F)}\uCH(\mathrm{M}_{G\cup i})\otimes \uCH(\mathrm{M}^G_F)[-1]\Big)\otimes \uJ_{\underline{\alpha}}(\mathrm{M}^F)[-1]\\
\cong & \bigoplus_{\substack{G \in \underline{\mathscr{C}}(\mathrm{M})\\ G \in \underline{\mathscr{S}}_i(\mathrm{M})}}  \  \uCH(\mathrm{M}_G\setminus i)\otimes \uJ_{\underline{\alpha}}(\mathrm{M}^G)[-1]\oplus \bigoplus_{\substack{G \in \underline{\mathscr{C}}(\mathrm{M})\\ G \in \underline{\mathscr{S}}_i(\mathrm{M})}}  \  \uCH(\mathrm{M}_G\setminus i)\otimes \uJ_{\underline{\alpha}}(\mathrm{M}^G)[-2] \ \oplus\\
& \bigoplus_{\substack{G\in \underline{\mathscr{S}}_i(\mathrm{M}) \\ F\in \underline{\mathscr{C}}(\mathrm{M}^G)}}\uCH(\mathrm{M}_{G\cup i})\otimes \uCH(\mathrm{M}^G_F)\otimes \uJ_{\underline{\alpha}}(\mathrm{M}^F)[-2]. 
\end{align*}
The decomposition \eqref{num_decomposition} follows. 
\end{proof}

\begin{remark}
The decomposition of graded vector spaces appearing in \cite[Theorem 6.18]{AHK}
specializes to decompositions of $\uCH(\mathrm{M})$ and of $\CH(\mathrm{M})$, where the latter goes through
Remark \ref{RemarkAlternative}.  At the level of Poincar\'e polynomials, these decompositions coincide with those of 
Theorem \ref{TheoremSimplexDecomposition}.  However, the subspaces appearing in the decompositions
are not the same.  In particular, the decompositions in \cite[Theorem 6.18]{AHK} are not orthogonal, and they are not compatible with the
$\underline{\mathrm{H}}_{\underline{\alpha}}(\mathrm{M})$-module structure on $\uCH(\mathrm{M})$
or the $\mathrm{H}_{\alpha}(\mathrm{M})$-module structure on $\CH(\mathrm{M})$.
\end{remark}

\bibliography{refs}
\bibliographystyle{amsalpha}

\end{document}